\documentclass[10pt]{amsart}
\usepackage{amsmath}
\usepackage{amssymb}
\usepackage{amscd}
\usepackage{amsthm}
\usepackage[centering]{geometry}
\usepackage{amsxtra}
\usepackage{mathabx}
\usepackage{xcolor}
\usepackage{scalerel,stackengine}

\definecolor{darkred}{rgb}{0.5,0,0}
\definecolor{darkgreen}{rgb}{0,0.5,0}
\definecolor{darkblue}{rgb}{0,0,0.5}
\usepackage[shortlabels]{enumitem}
\setlist[enumerate,1]{label=(\roman*)}

\usepackage{microtype}
\usepackage[linktocpage,pdftex]{hyperref}
\hypersetup{linktocpage
           ,pdfborder={0 0 0}
           ,colorlinks
           ,linkcolor=darkblue
           ,filecolor=darkgreen
           ,urlcolor=darkred
           ,citecolor=darkblue
           }

\stackMath
\newcommand\reallywidehat[1]{%
\savestack{\tmpbox}{\stretchto{%
  \scaleto{%
    \scalerel*[\widthof{\ensuremath{#1}}]{\kern-.6pt\bigwedge\kern-.6pt}%
    {\rule[-\textheight/2]{1ex}{\textheight}}
  }{\textheight}%
}{0.5ex}}%
\stackon[1pt]{#1}{\tmpbox}%
}

\numberwithin{equation}{section}

\DeclareMathOperator{\supp}{supp}
\DeclareMathOperator{\dens}{dens}

\DeclareMathOperator{\lin}{span}
\providecommand{\cl}[1]{\mathrm{cl}(#1)}
\providecommand{\inn}[1]{\mathrm{int}(#1)}
\newcommand{\R}{{\mathbb R}}

\newcommand{\C}{{\mathbb C}}

\newcommand{\N}{\mathbb N}

\newcommand{\cd}{{\mathcal D}}

\newcommand{\cA}{{\mathcal A}}
\newcommand{\cL}{{\mathcal L}}

\renewcommand{\hat}{\widehat}

\newcommand{\mcard}{\mathrm{card}}
\renewcommand{\mid}{:}

\RequirePackage{ifthen}
\DeclareMathOperator{\PW}{\mathcal{PW}}
\DeclareMathOperator{\PWd}{\mathcal{PW}^\prime\!}
\DeclareMathOperator{\CS}{\mathcal{S}}

\providecommand{\cl}[1]{\cl{#1}}
\providecommand{\rstr}[1]{|_{#1}}
\providecommand{\spt}{\supp}
\providecommand{\from}{:}

\providecommand{\qtextq}[1]{\quad\text{#1}\quad}

\providecommand{\isect}{\cap}
\providecommand{\Union}{\bigcup}

\RequirePackage{mathtools}
\DeclarePairedDelimiter{\abs}{\lvert}{\rvert}
\providecommand{\norm}[2][]{\lVert#2\rVert\ifthenelse{\equal{}{#1}}{}{_{#1}}}
\providecommand{\bignorm}[2][]{\bigl\lVert#2\bigr\rVert\ifthenelse{\equal{}{#1}}{}{_{#1}}}
\providecommand{\Bignorm}[2][]{\Bigl\lVert#2\Bigr\rVert\ifthenelse{\equal{}{#1}}{}{_{#1}}}
\providecommand{\biggnorm}[2][]{\biggl\lVert#2\biggr\rVert\ifthenelse{\equal{}{#1}}{}{_{#1}}}
\providecommand{\Biggnorm}[2][]{\Biggl\lVert#2\Biggr\rVert\ifthenelse{\equal{}{#1}}{}{_{#1}}}
\providecommand{\spr}[3][]{\langle#2,#3\rangle\ifthenelse{\equal{}{#1}}{}{_{#1}}}
\providecommand{\bigspr}[3][]{\bigl\langle#2,#3\bigr\rangle\ifthenelse{\equal{}{#1}}{}{_{#1}}}
\providecommand{\Bigspr}[3][]{\Bigl\langle#2,#3\Bigr\rangle\ifthenelse{\equal{}{#1}}{}{_{#1}}}
\providecommand{\biggspr}[3][]{\biggl\langle#2,#3\biggr\rangle\ifthenelse{\equal{}{#1}}{}{_{#1}}}
\providecommand{\Biggspr}[3][]{\Biggl\langle#2,#3\Biggr\rangle\ifthenelse{\equal{}{#1}}{}{_{#1}}}
\let\originald\d 
\renewcommand{\d}{\ifthenelse{\boolean{mmode}}{\mathrm d}{\originald}}
\providecommand{\dd}{\,\d}
\let\originali\i 
\renewcommand{\i}{\ifthenelse{\boolean{mmode}}{\mathrm i}{\originali}}
\providecommand{\ifu}[1]{\mathbf1_{#1}}

\providecommand{\xto}{\xrightarrow}

\newtheorem{theorem}{Theorem}[section]
\newtheorem{lemma}[theorem]{Lemma}
\newtheorem{prop}[theorem]{Proposition}

\newtheorem{fact}[theorem]{Fact}

\newtheorem{defi}[theorem]{Definition}

\newtheoremstyle{rremark}%
       {1.8ex\@plus1ex}     
       {2.1ex\@plus1ex\@minus.5ex} 
       {\normalfont}        
       {0pt}                
       {\bfseries}          
       {.}                  
       {.5em}               
       {}                   

\theoremstyle{rremark}
\newtheorem{remark}[theorem]{Remark}

\begin{document}
\title[Sampling and interpolation]{On sampling and interpolation by model sets}

\author{Christoph Richard}
\address{Department f\"{u}r Mathematik, Friedrich-Alexander-Universit\"{a}t Erlangen-N\"{u}rnberg,
Cauerstrasse 11, 91058 Erlangen, Germany}
\email{christoph.richard@fau.de}

\author{Christoph Schumacher}
\address{Fakult\"at f\"ur Mathematik, Technische Universit\"at Dortmund,
Vogelpothsweg 87, 44227 Dortmund, Germany}
\email{christoph.schumacher@mathematik.tu-dortmund.de}

\begin{abstract}
We refine a result of Matei and Meyer \cite{mm10} on stable sampling and stable interpolation for simple model sets. Our setting is model sets in locally compact abelian groups and Fourier analysis of unbounded complex Radon measures as developed by Argabright and de Lamadrid \cite{ARMA1}. This leads to a refined version of the underlying model set duality between sampling and interpolation. For rather general model sets, our methods also yield an elementary proof of stable sampling and stable interpolation sufficiently far away from the critical density, which is based on the Poisson Summation Formula.
\end{abstract}

\subjclass[2010]{43A25 (primary), and 52C23 (secondary)}

\keywords{sampling, interpolation, model set, Poisson summation formula}

\maketitle

\section{Background and plan of the article}

Sampling concerns the problem of reconstructing a function~$f$ from  its restriction~$f|_\Lambda$ to a  subdomain~$\Lambda$. Interpolation concerns the question of how to extend a function defined on the subdomain. For particular function spaces, both problems are classical in harmonic analysis. We are interested in irregular sampling domains arising from model sets in locally compact abelian groups.

\subsection{Function spaces}

Consider a locally compact abelian (LCA) group~$G$ and its dual group~$\widehat G$. Later, we will sometimes require~$G$ to be $\sigma$-compact or metrizable.
Note that $\sigma$-compactness and metrisability are dual notions \cite[(24.48)]{HeRo}, and that $\sigma$-compactness and metrisability is equivalent to second countability \cite[Chap.~IX, §2.9, Cor]{Bou}. As a consequence, second countability of~$G$ is equivalent to second countability of $\widehat G$.
Fix Haar measures~$\theta_G$ on~$G$ and~$\theta_{\widehat G}$ on~$\widehat G$ such that the Fourier inversion theorem \cite[Thm.~4.4.5]{Rei2} holds.  Recall that for integrable $f:G\to \mathbb C$, its Fourier transform is defined by $\widehat f(\chi)=\int f(x) \overline{\chi(x)}\dd\theta_G(x)$ for $\chi\in\widehat G$. We also write $\widecheck f(\chi)=\int f(x) \chi(x)\dd\theta_G(x)$ for the inverse Fourier transform. 

\begin{defi}[Paley--Wiener space]
  Let~$G$ be an LCA  group.
  Consider $K\subset \widehat G$ to be relatively compact and measurable.
  Then the translation invariant vector space of functions with restricted band width
  \begin{equation*}
    \PW_K=\{f\in L^2(G) \mid  \widehat f |_{K^c}=0 \text{ almost everywhere} \}
  \end{equation*}
  is called the \emph{Paley--Wiener space} with respect to~$K$.
\end{defi}

\begin{remark}\leavevmode
\begin{itemize}
\item[(i)] We consider relatively compact~$K$ in this article only. Then~$G$ may be assumed to be metrizable. Indeed, for relatively compact~$K$ we may assume that~$\widehat G$ is compactly generated without loss of generality, see \cite[Sec.~8]{GKS08}. 
\item[(ii)] Note that $f\in \PW_K$ satisfies $\widehat f\in L^1(\widehat G)\cap L^2(\widehat G)$.
  Indeed, $\widehat f\in L^1(\widehat G)$ follows from H\"older's inequality by assumption on~$K$.
We thus have $\widecheck{\widehat{f}}\in C_0(G)$ and $f=\widecheck{\widehat{f}}$ in $L^2(G)$ by the Fourier inversion theorem. Hence any  equivalence class $f\in \PW_K$ has a unique continuous representative. In this article, we will always consider $\PW_K$ as a space of continuous functions.
\item[(iii)] Note that $(\PW_K, \norm{\cdot}_{L^2(G)})$ is a Hilbert space. Closedness is obvious from continuity of the map $f\mapsto \norm{\widehat f|_{K^c}}_{L^2(\widehat G)}^2$. The above definition is also used in \cite{KL11, AACB16}. Another definition is used in \cite{mm10, GL16}. If~$K$ is closed in $\widehat G$, then both definitions are equivalent. This is discussed in Section~\ref{sec:closed}.
\end{itemize}
\end{remark}

\medskip

The following subspace of~$\PW_K$ will be central in our subsequent analysis, as it provides continuous test functions for complex Radon measures which are transformable in the sense of Argabright and de Lamadrid \cite{ARMA1, RS15}. It is defined to be
  \begin{equation*}
    \CS_K=\PW_K\cap L^1(G)=\{f\in L^1(G) :  \widehat f |_{K^c}=0 \} \,.
  \end{equation*}
In this article, we consider~$\CS_K$ as a space of continuous functions. 
Any $f\in \CS_K$ satisfies $f\in L^1(G)\cap L^2(G)$ and $\widehat f\in C_c(\widehat G)$. 
Note that~$\CS_K$ is dense in $(\PW_K, \norm{\cdot}_{L^2(G)})$ if~$K$ is Riemann measurable, see Proposition~\ref{prop:PWdense}.

\subsection{Stable sampling and stable interpolation}

The stable sampling condition ensures that any (continuous) Paley--Wiener function is reconstructable in a stable way from its values on a given uniformly discrete point set. We say that $\Lambda\subset G$ is uniformly discrete if there exists a compact zero neighborhood $B\subset G$ such that $\mcard(\Lambda \cap (x+B))\le 1$ for all $x\in G$.

\begin{defi}[stable sampling]\label{def:StSa}
  Let~$G$ be an LCA group and let $K\subset\widehat G$ be relatively compact and measurable.
  A uniformly discrete set $\Lambda\subset G$ is \emph{stable sampling} for $\PW_K$ if and only if there exist positive finite constants $A,B$
  such that
  \begin{equation}\label{eq:stablesampling}
    A\norm{f}_{L^2(G)}^2
    \le \norm{f|_{\Lambda}}_{\ell^2(\Lambda)}^2
    \le B\norm{f}_{L^2(G)}^2
  \end{equation}
for all $f\in\PW_K$.
\end{defi}
 If~$\Lambda$ is stable sampling for~$\PW_K$, then any $f\in\PW_K$ is uniquely determined by its values on $\Lambda$. Indeed, for $f_1,f_2\in\PW_K$ satisfying ${f_1}|_\Lambda={f_2}|_\Lambda$ we have $f=f_1-f_2\in\PW_K$, too. Thus $\norm{f}_{L^2(G)}\le0$, which implies $f=0$ pointwise by continuity of~$f$.

\medskip

For a uniformly discrete set~$\Lambda$, stable interpolation for~$\PW_K$ may be described as follows: For every square integrable sequence $\varphi\in\ell^2(\Lambda)$ there exists at least one $f\in\PW_K$ such that $f=\varphi$ on~$\Lambda$. As discussed in Remark~\ref{rem:exponential}, this interpolation property is rephrased in the following definition.  Let us denote by~$\ifu S$ the indicator function of the set~$S$. 

\begin{defi}[stable interpolation]\label{def:si}
  Let~$G$ be an LCA group and let $K\subset\hat G$ be relatively compact and measurable.
  A uniformly discrete set $\Lambda\subset G$ is \emph{stable interpolating} for $\PW_K$  if and only if there exist positive finite constants $A,B$
  such that
  \begin{equation}\label{eq:si}
    A\norm\varphi_{\ell^2(\Lambda)}^2
    \le
      \norm{\ifu Kf_\varphi}_{L^2(\widehat G)}^2
    \le B\norm\varphi_{\ell^2(\Lambda)}^2
  \end{equation}
for all finitely supported $\varphi\in \ell^2(\Lambda)$.
Here $f_\varphi\in C(\widehat G)$ is the discrete Fourier transform of $\varphi$ given by
  $f_\varphi(\chi)=\sum_{\lambda\in\Lambda}\varphi(\lambda)\overline{\chi(\lambda)}$.

\end{defi}

\begin{remark}\label{rem:exponential}
As is well known, stable sampling and stable interpolating can be rephrased using Hilbert space techniques \cite{Y01}. Let us briefly indicate this connection in order to clarify the above definitions. For relatively compact and measurable $K\subset \widehat G$ consider the Hilbert space $L^2(K)$, compare Lemma~\ref{prop:L2closed}. The Fourier transform provides a norm preserving isomorphism between $\PW_K$ and $L^2(K)$. For $x\in G$ consider the ``exponential'' $e_x: \widehat G\to \mathbb C$ defined by $e_x(\chi):=\chi(x)$ for all $\chi\in \widehat G$.
If $\langle \cdot, \cdot \rangle$ denotes the usual scalar product in $L^2(K)$, we then have for any $f\in \PW_K$ and any $x\in G$ that
\begin{displaymath}
\langle \widehat f, e_{-x}|_K\rangle = \int_K \widehat f(\chi) \overline{e_{-x}|_K}(\chi) \, {\rm d} \theta_{\widehat G}(\chi)= \int \widehat f(\chi) \chi(x) \, {\rm d} \theta_{\widehat G}(\chi)=\widecheck{\widehat f}(x)=f(x) \ .
\end{displaymath}
For a uniformly discrete set $\Lambda\subset G$ consider $E(\Lambda)\subset L^2(K)$ defined by
\begin{displaymath}
E(\Lambda)=\{ e_{-\lambda}|_{K}: \lambda\in \Lambda\} \ .
\end{displaymath}
The stable sampling inequalities say that~$E(\Lambda)$ is a so-called frame in~$L^2(K)$. In our setting frame theory provides a reconstruction algorithm for $\widehat f\in L^2(K)$, see e.\,g.~\cite[p.~159]{Y01}.
The interpolation property is equivalent to requiring that~$E(\Lambda)$ is a so-called Riesz-Fisher sequence. 
Hilbert space arguments for Riesz-Fisher sequences, see e.\,g.~\cite[p.~129]{Y01}, can be used to infer the following: The interpolation property is equivalent to the existence of a lower inequality  in  Eqn.~\eqref{eq:si}. Moreover, existence of an upper inequality  in Eqn.~\eqref{eq:si} is equivalent to the existence of an upper inequality  in  Eqn.~\eqref{eq:stablesampling}. Is is well known that the upper inequalities hold for uniformly discrete sets~$\Lambda$. We will give an elementary proof of the upper inequality for interpolation in Lemma~\ref{lem:iub}.
\end{remark}

\subsection{Landau's necessity conditions}

In Euclidean space, necessary conditions for~$\Lambda$ to be a set of stable sampling or of stable interpolation for~$\PW_K$ have been given by Landau \cite{L67}. These have been extended to general LCA groups in \cite{GKS08}. The conditions are formulated in terms of a particular uniform density.

\smallskip

In order to compute asymptotic frequencies, one often uses suitable averaging sequences generalizing balls, whose ``boundary to bulk ratio'' vanishes asymptotically. We will work with so-called van Hove sequences.
For their definition, consider $A\subset G$ and a compact set $K\subset G$.  We denote the topological closure and interior of a set~$A$ by $\cl A$ and $\inn A$, respectively.
 The \emph{$K$-boundary} of~$A$ is given by
\begin{displaymath}
\partial^K\!A = [(K+\cl{A}) \cap \cl{A^c}] \cup [(K+\cl{A^c})\cap \cl{A}] \, .
\end{displaymath}
Note that~$\partial^K\!A$ is compact if~$A$ is relatively compact.
As $\partial A=\partial^{\{e\}}A\subset \partial^U\!A$ for~$U$ any compact zero neighborhood, the van Hove boundary may be considered as a thickened topological boundary in that case.
The value $\theta_G(\partial^U \! A)$, where $\theta_G$ is the chosen Haar measure on~$G$ and~$A$ is relatively compact, can be chosen arbitrarily close to $\theta_G(\partial A)$ for sufficiently concentrated zero neighborhoods~$U$, compare \cite[Lem.~4.4]{HR15}.
A sequence $(A_n)_{n\in\mathbb N}$ of compact sets in~$G$ of positive Haar measure is \emph{van Hove} if and only if
\begin{displaymath}
\frac{\theta_G(\partial^K \!A_n)}{\theta_G(A_n)} \to 0 \qquad (n\to\infty) \ ,
\end{displaymath}
for every compact set~$K$ in~$G$. Van Hove sequences have been introduced by Schlottmann \cite{Martin2} for model set analysis, where also their existence in $\sigma$-compact LCA groups is shown.
In Euclidean space, any sequence of nonempty compact rectangular boxes of diverging positive inradius is a van Hove sequence, as well as any sequence of compact nonempty balls of diverging positive radius is a van Hove sequence.
In general, note that the van Hove property is stable under shifting, i.\,e., for a van Hove sequence $(A_n)_{n\in\mathbb N}$ the sequence $(A_n+x_n)_{n\in\mathbb N}$ is still van Hove, for any shift sequence $(x_n)_{n\in\mathbb N}$. Further properties are collected in Section~\ref{sec:vHBd}.  For additional background, see e.\,g.~\cite{LR, mr13, HR15} and references therein.

\begin{defi}[Banach densities]
Let~$\Lambda$ be a uniformly discrete set in  a $\sigma$-compact LCA group~$G$. Let $\cA=(A_n)_{n\in\mathbb N}$ be any van Hove sequence in~$G$. We define the lower and upper Banach densities
\begin{displaymath}
\begin{split}
\cd^-_\cA(\Lambda)&= \liminf_{n\to\infty} \inf_{t\in G} \frac{1}{\theta_G(A_n)}\mcard(\Lambda\cap (A_n+t)) \,, \\
\cd^+_\cA(\Lambda)&= \limsup_{n\to\infty} \sup_{t\in G}\frac{1}{\theta_G(A_n)}\mcard(\Lambda\cap (A_n+t))\,.
\end{split}
\end{displaymath}
If $\cd^-_\cA(\Lambda)=\cd^+_\cA(\Lambda)=:\cd(\Lambda)$, we say that~$\Lambda$ has Banach density~$\cd(\Lambda)$.
\end{defi}

The above densities are finite due to uniform discreteness of~$\Lambda$.
In $G=\mathbb R^d$ and with centered $n$-balls~$A_n$, the above densities are also called Landau-Beurling densities,  compare \cite[Sec.~7]{GKS08}. In ergodic theory or number theory, the attribution to Banach is more common. Note that existence of a Banach density implies independence of the choice of the van Hove sequence.  Banach densities exist for any regular model set by  Fact~\ref{thm:df2}, see Section~\ref{mslcs} for definitions. In particular any \emph{lattice}~$L$ in~$G$, i.\,e., any discrete co-compact subgroup of~$G$, has a Banach density, which equals the inverse measure of a measurable fundamental domain. For countable LCA groups and uniformly discrete~$\Lambda$, it has recently been shown that the lower and upper Banach densities exist as a limit \cite[Sec.~2.2]{DHZ19}.

\medskip

We rephrase Theorem 1'' from \cite{GKS08} using upper and lower Banach densities. This is justified by Lemma~\ref{lem:BanvH}.

\begin{fact}[Landau's necessity conditions]\label{thm:landau}
Let~$G$ be a $\sigma$-compact LCA group.  Let $K\subset \widehat G$ be a relatively compact measurable set. Let $\Lambda\subset G$ be a uniformly discrete set and let~$\cA$ be any van Hove sequence in~$G$. Then the following hold.
\begin{enumerate}
\item[a)] If~$\Lambda$ is stable sampling for~$\PW_K$, then $\cd^-_\cA(\Lambda)\ge \theta_{\widehat G}(K)$.
\item[b)] If~$\Lambda$ is stable interpolating for~$\PW_K$, then $\cd^+_\cA(\Lambda)\le  \theta_{\widehat G}(K)$.
\qed
\end{enumerate}
\end{fact}

\subsection{Sampling and interpolation near the critical density}

An interesting problem concerns the question which point sets admit stable sampling or stable interpolation arbitrarily close to the above critical density value $\theta_{\widehat G}(K)$. A classic result for subsets of the line traces back to works of Ingham, Beurling and Kahane. See \cite[Thm.~4.3, Thm.~4.5]{mm10} for the statement and a discussion.

\begin{fact}[Stable sampling and stable interpolation on the line]\label{fact:1d}
Let $I\subset \mathbb R$ be a bounded interval, and let $\Lambda\subset \mathbb R$ be a uniformly discrete set. Let~$\mathcal A$ be any van Hove sequence of intervals in~$\mathbb R$.

\begin{itemize}
\item[(i)] If $\mathcal D_{\mathcal A}^-(\Lambda)>\abs{I}$, then~$\Lambda$ is stable sampling for~$\PW_I$.
\item[(ii)]  If $\mathcal D_{\mathcal A}^+(\Lambda)<\abs{I}$, then~$\Lambda$ is stable interpolating for~$\PW_I$.
\qed
\end{itemize}
\end{fact}

For so-called simple model sets in Euclidean space, a corresponding result has been proved by Matei and Meyer \cite[Thm.~3.2]{mm10}. For general LCA groups, sampling and interpolation sets arbitrarily close to the critical density  have been constructed in \cite{AAC15}. For simple model sets in general LCA groups, an extension of \cite{mm10} has been suggested by Agora et al.~\cite[Thm.~6]{AACB16}. As emphasized in \cite{KL11, GL16}, these  results rely on a duality between stable sampling and stable interpolation for model sets having Euclidean internal space. This duality allows to transfer the general problem to the line, where Fact~\ref{fact:1d} can be invoked.

\smallskip

In  Theorem~\ref{thm:main}, we will formulate the sampling and interpolation duality for model sets in the framework of LCA groups, where we use Fourier analysis of unbounded complex Radon measures as developed by Argabright and de Lamadrid \cite{ARMA1}. This setting is well suited, since integrable Paley--Wiener functions are test functions for transformable measures. In our proof, we follow the constructions of Matei and Meyer but argue without Bruhat-Schwartz functions, which are used in \cite{AACB16}.  A central part of the proof relies on mathematical diffraction theory as described in e.\,g.~\cite{MoSt, RS15}, see also \cite{RS17} for a recent review. In comparison to previous results, our approach yields additional structural insight, which results in a refined version of the duality theorem in a more general setting. See Section~\ref{sec:statedis} for further discussion.

\smallskip

As an application, we obtain the following version of the stable sampling and stable interpolation theorem for simple model sets. Cut-and-project schemes and model sets will be defined in the following section. Note that the model sets in the theorem below have a Banach density.

\begin{theorem}[Stable sampling and stable interpolation for simple model sets]\label{cor:SIn1}
Let  $(G,H,\cL)$ be a complete cut-and-project scheme, where~$G$ is a $\sigma$-compact LCA group and where $H=\mathbb R$.
Let $I\subset H$ be a bounded nonempty interval and let $\Lambda_I\subset G$ be the associated model set.  Assume that $K\subset \widehat G$ is relatively compact and measurable.  Then the following hold.
  \begin{enumerate}
 \item[(i)] If $\mathcal D(\Lambda_I)>\theta_{\widehat G}(\cl K)$, then~$\Lambda_{I}$ is stable sampling for~$\PW_K$.

   \item[(ii)] If $\mathcal D(\Lambda_I)<\theta_{\widehat G}(\inn K)$, then~$\Lambda_I$ is stable interpolating for~$\PW_{K}$.
  \end{enumerate}
\end{theorem}

\begin{remark}\leavevmode
\begin{itemize}
\item[(i)] Let us compare the above result to \cite[Thm.~3.2]{mm10} and to \cite[Thm.~6]{AACB16}. In those references, it is assumed that~$K$ is compact, and for~$(ii)$ it is assumed in addition that~$K$ is Riemann measurable.  However these assumptions are not necessary, as may be seen by approximation arguments using regularity of the Haar measure.
\item[(ii)] We can slightly relax the completeness assumption on the cut-and-project scheme, see Remark~\ref{rem:d1} and the proof of Theorem~\ref{cor:SIn1}: From the four projection assumptions in the complete cut-and-project scheme $(G,H,\cL)$ described in Definition~\ref{def:cps}, the proof of (i) does not use that~$\cL$ projects densely to~$H$. The proof of (ii) only uses that~$\cL$ projects densely to~$G$ and injectively to~$H$.
\item[(iii)] The works \cite{KL11} and \cite{GL16} construct model sets~$\Lambda_I$ in Euclidean space which are simultaneously stable sampling and stable interpolating. This is possible if~$\Lambda_I$ has critical Banach density $\theta_{\widehat G}(K)$, and  if~$K$ is Riemann measurable  and a so-called  bounded remainder set. In particular in \cite[Lem.~2.1]{KL11} and \cite[Thm.~3.1]{GL16}, the arguments leading to \cite[Thm.~3.2]{mm10} are reviewed.
\end{itemize}
\end{remark}

In fact, our methods also allow to prove stable sampling and stable interpolation for general weak model sets sufficiently far away from the critical density. This will be done in the first part of the article, see Proposition~\ref{lem:slb} and Proposition~\ref{lem:ilbms}. That part will also serve as a preparation for the second part of the article, as some of the corresponding arguments will reappear in the proof of the duality theorem.

\medskip

\noindent \textit{Note added in proof:}
After submitting this manuscript, we became aware of recent work by Meyer  \cite{me18}, where he reviews his constructions for simple model sets in Euclidean space and gives an alternative proof of  \cite[Thm.~3.2]{mm10}, again by reduction to the line. Meyer's proof uses a duality principle from \cite{me73}. It differs from the Duality Theorem~\ref{thm:main}, which is based on square integrable Paley-Wiener functions.

\subsection{Outline of the article} In Section~\ref{mslcs}, we review weighted model combs and introduce our notation. 
Section~\ref{sec:sims} is devoted to a proof of stable interpolation for weak model sets away from the critical density. 
Section~\ref{sec:samp} is devoted to a proof of stable sampling for weak model sets away from the critical density. 
Although we discuss in Section~\ref{sec:cdlb} that our results are not optimal, we believe that our techniques, which presently rely on standard estimates, may be substantially refined. Section~\ref{sec:densf} discusses density formulae for model sets. This is a preparation for the proof of the duality theorem, which uses a density formula with smooth averaging functions instead of van Hove sequences. Section~\ref{sec:dualityms} contains a proof of the duality theorem for stable sampling and stable interpolation. We conclude that section with a proof of Theorem~\ref{cor:SIn1}. The last sections collect some facts about Paley--Wiener spaces, van Hove sequences and Banach densities.

\subsection{Acknowledgments}
The results of this article emerged from discussions during mutual visits of the authors at TU Chemnitz, FAU Erlangen-N\"urnberg and TU Dortmund. The authors are grateful for support by their faculties.
CS would like to thank Thomas Kalmes and Albrecht Seelmann for enlightening and fruitful discussions.
CR would like to thank the participants of the conference \textit{Model sets and Aperiodic Order} in Durham in September 2018, where parts of the above results were presented, for inspiring discussions. We thank the referees for very helpful comments and suggestions.

\section{Weighted model sets}\label{mslcs}

In order to set up our notation, we recall relevant notions related to model sets, following the monograph \cite{BG2} where possible.

\subsection{Cut-and-project schemes and weak model sets}

\begin{defi}[Cut-and-project scheme]\label{def:cps}
  Let $G,H$ be LCA groups, and let $\cL\subset G\times H$ be a lattice, i.\,e., a discrete co-compact subgroup in $G\times H$.  We then call $(G,H,\cL)$ a \emph{bare cut-and-project scheme}. If~$\cL$ projects injectively to~$G$ and densely to~$H$, then $(G,H,\cL)$ is called a \emph{cut-and-project scheme}. If~$\cL$ projects injectively and densely to both~$G$ and~$H$, we call $(G,H,\cL)$ a \emph{complete cut-and-project scheme}.
\end{defi}

\begin{remark}
For a (bare) cut-and-project scheme $(G,H,\cL)$, we will denote by $\pi_G\from G\times H\to G$ and $\pi_H\from G\times H\to H$ the canonical projections.
Denseness of $\pi_H(\cL)$ can always be obtained by passing from~$H$ to the LCA group $H'=\cl{\pi_H(\cL)}$. Injectivity of $\pi_G|_\cL$ allows to identify~$\cL$ and $\pi_G(\cL)$. This identification will repeatedly be used below.
The name \textit{complete cut-and-project scheme} appears in \cite{AACB16}, where also some structure theory is developed.
\end{remark}

Model sets in~$G$ are obtained by ``cutting from~$\cL$ all points inside some strip parallel to~$G$'' and projecting them to~$G$. An analogous construction can be used to obtain point sets in~$H$.
This is the content of the following definition.

\begin{defi}[Weak model set]
Assume that $(G,H,\cL)$ is a bare cut-and-project scheme. Take a \emph{window} $W\subset H$ and consider the \emph{projection set} $\Lambda_W\subset G$ defined by
  \begin{equation*}
    \Lambda_W=\pi_G(\cL\cap (G\times W)) \,.
  \end{equation*}
If $(G,H,\cL)$ is a cut-and-project scheme, then

\begin{itemize}

\item[(i)] $\Lambda_W$ is called a \emph{weak model set} if~$W$ is relatively compact.

\item[(ii)] $\Lambda_W$ is called a \emph{model set} if~$\Lambda_W$ is a weak model set and if~$W$ has nonempty interior.

\item[(iii)]  $\Lambda_W$ is called a \emph{regular model set} if~$\Lambda_W$ is a model set and if~$W$ is Riemann measurable, i.\,e.,  $\theta_H(\partial W)=0$, where~$\theta_H$ is a Haar measure on~$H$.
\end{itemize}
We may take a window $V\subset G$ and consider the projection set $_V\Lambda\subset H$ defined by
  \begin{equation*}
   _V \Lambda=\pi_H(\cL\cap (V\times H))\,.
  \end{equation*}
If $(G,H,\cL)$ is a complete cut-and-project scheme, we call~$_V\Lambda$ a (weak, regular) model set depending on the corresponding properties of~$V$.
\end{defi}

\begin{remark}[Elementary properties of projection sets]
Any weak model set is uniformly discrete. More generally,
any projection set with relatively compact window is uniformly discrete, compare to \cite[Prop.~2.6(i)]{RVM3}. Any model set is relatively dense. The proof of the latter statement uses that~$\cL$ projects densely to~$H$, compare the proof of \cite[Prop.~2.6(i)]{RVM3}. For any regular model set in a $\sigma$-compact group~$G$, all of its pattern frequencies exist as a limit which is uniform in shifts of the chosen averaging sequence, see also below. Projection sets~$\Lambda$ with relatively compact window are of \emph{finite local complexity}, i.\,e., we have $\Lambda-\Lambda\subset \Lambda+F$, where~$F$ is some finite set. One also says that~$\Lambda$ is an \emph{almost lattice}.
\end{remark}

\begin{remark}[Star map]
If $(G,H,\cL)$ is a bare cut-and-project scheme and if~$\cL$ projects injectively to~$G$, then the \emph{star map}
\begin{equation*}
    s\from\pi_G(\cL)\to\pi_H(\cL)
    \qtextq{such that}
    (g,s(g))\in\cL
\end{equation*}
is well defined. Hence we can write $\Lambda_W=\{g\in\pi_G(\cL):s(g)\in W\}$ in that case.
If $(G,H,\cL)$ is a complete cut-and-project scheme, then the star map is a bijection. Hence in that case we may write $_V \Lambda=\{h\in\pi_H(\cL):s^{-1}(h)\in V\}$.
\end{remark}

\subsection{Dual cut-and-project scheme}\label{sec:dual}

We recall duality for cut-and-project schemes, compare \cite{RVM3}.
Denote by $\cL_0\subset \hat G\times\hat H$ the annihilator of~$\mathcal L$, i.\,e.,
 \begin{equation*}
    \cL_0:=\{\ell_0\in\hat G\times\hat H:
    \ell_0(\ell)=1\text{ for all } \ell\in\cL\}\, .
\end{equation*}
By Pontryagin duality,  is a~$\cL_0$ is lattice in $\widehat G\times \widehat H$, $\pi_G\rstr\cL$ is one-to-one iff $\pi_{\hat H}(\cL_0)$ is dense in~$\hat H$, and $\pi_H(\cL)$ is dense in~$H$ iff $\pi_{\hat G}\rstr{\cL_0}$ is one-to-one.
Hence we can define

\begin{lemma}[Dual cut-and-project scheme]
  Let $(G,H,\cL)$ be a bare cut-and-project scheme. Then $(\hat G,\hat H,\cL_0)$ is a bare cut-and-project scheme, which we call  the dual cut-and-project scheme.
  If  $(G,H,\cL)$ is a (complete) cut-and-project scheme, then $(\hat G,\hat H,\cL_0)$ is a (complete) cut-and-project scheme as well.
  \qed
\end{lemma}

We write
  $\hat s\from\pi_{\hat G}(\cL_0)\to\pi_{\hat H}(\cL_0)$ for the star map in the dual cut-and-project scheme $(\hat G,\hat H,\cL_0)$.
In the following, given a relatively compact subset $K\subset \widehat G$, we will frequently consider the projection set
\begin{displaymath}
_K \Lambda=\pi_{\widehat H}\left(\cL_0\cap(K\times \widehat H)\right) \, .
\end{displaymath}

\subsection{Weighted model combs and their transforms}

Instead of working with point sets in~$G$, it is sometimes advantageous to consider associated weighted Dirac combs. Let us denote by~$\mathcal M(G)$ the space of complex Radon measures on~$G$. Assume that $(G,H,\cL)$ is a bare cut-and-project scheme and that $h:H\to\mathbb C$ is a bounded function such that
\begin{displaymath}
\omega_h=\sum_{(x,y)\in\cL} h(y)\delta_x
\end{displaymath}
is a complex Radon measure on~$G$, i.\,e., we have $\omega_h\in \mathcal M(G)$. Then~$\omega_h$ is called a \emph{weighted model comb}, and~$h$ is called its \emph{weight function}. Examples of weight functions are $h\in C_c(H)$ and $h=\ifu W$ for relatively compact $W\subset H$, see the proof of Lemma 4.8 in \cite{RS17}. We have $\Lambda_W=\supp(\omega_{\ifu W})$, where~$\supp$ denotes the measure-theoretic support. Another example is $h\in C_0(H)$ of sufficiently fast decay \cite{LR}.

We will also consider weighted model combs $_g\omega=\sum_{(x,y)\in\cL} g(x)\delta_y$ for bounded $g:G\to\mathbb C$ such that $_g\omega\in \mathcal M(H)$.  For relatively compact $V\subset G$, we have $_V \Lambda=\supp(_{\ifu V}\omega)$.

Consider the associated dual cut-and-project scheme $(\widehat G, \widehat H, \cL_0)$. As no confusion may arise, we will denote weighted model combs in that scheme by the same symbol, i.\,e., we write
\begin{displaymath}
\omega_\psi=\sum_{(\chi,\eta)\in\cL_0} \psi(\eta)\delta_\chi
\end{displaymath}
for bounded $\psi:\widehat H\to\mathbb C$. If $\psi=\ifu \Omega$ for $\Omega\subset \widehat H$, we have $\Lambda_\Omega=\supp(\omega_{\ifu \Omega})$.

Recall that $\mu\in \mathcal M(G)$ is \emph{translation bounded} if $\sup_{t\in G}\abs\mu(K+t)<\infty$ for every compact $K\subset G$, where~$\abs\mu$ is the total variation measure of~$\mu$, see e.\,g.~\cite[Ch.~1]{ARMA1}. Let us denote by $\mathcal M^\infty(G)$ the space of translation bounded complex Radon measures on~$G$. Observing that  $\abs{\omega_h}=\omega_{\abs h}$, we conclude $\omega_h\in\mathcal M^\infty(G)$ for $h\in C_c(H)$ or for $h=\ifu W$ with relatively compact $W\subset H$.
We write $L^1(\mu)$ for the space of $\abs\mu$-integrable functions.

\medskip

As weighted model combs are complex Radon measures, one might apply Fourier analysis of complex Radon measures as developed by Argabright and de Lamadrid \cite{ARMA1, MoSt, RS17, RS15}. Weighted model sets satisfy a Poisson Summation Formula for a large class of test functions.  Consider the function spaces
\begin{displaymath}
KL(G)=\{f\in C_c(G): \widehat f\in L^1(\widehat G)\}, \qquad
LK(G)=\{f\in L^1(G)\cap  C_0(G): \widehat f\in C_c(\widehat G)\}
\end{displaymath}
and note $\CS_K\subset LK(G)$ for all measurable relatively compact $K\subset \widehat G$. We have the following result.

\begin{lemma}[Transform of weighted model combs]\label{lem:FTWMS}
Let $(G,H,\cL)$ be a bare cut-and-project scheme. Assume that $h\in KL(H)\cup LK(H)$. Then both~$\omega_h$ and~$\omega_{\widecheck h}$ are translation bounded complex Radon measures. Moreover the Poisson Summation Formula
\begin{equation}\label{form:PSF}
\omega_h(g)=\dens(\cL)\cdot \omega_{\widecheck h}(\widecheck g) 
\end{equation}
holds for every $g\in KL(G)\cup LK(G)$.
\end{lemma}

\begin{remark}
Here~$\dens(\cL)$ is the reciprocal Haar measure of a measurable fundamental domain of~$\cL$. It coincides with the density of lattice points, evaluated on any van Hove sequence. Note that $\dens(\cL)\cdot \dens(\cL_0)=1$.
\end{remark}

\begin{proof}[Proof of Lemma~\ref{lem:FTWMS}]
Note that the Poisson Summation Formula~\eqref{form:PSF} holds for all $g\in KL(G)$ and $h\in KL(H)$ by \cite[Thm.~4.10~(ii)]{RS15}. In fact, the same formula also holds for $g\in LK(G)$ and $h\in KL(H)$. Indeed, as~$\omega_h$ is twice transformable by \cite[Thm.~4.12]{RS15} and $\widehat g\in KL(\widehat G)$ by the Fourier inversion theorem, we have
\begin{displaymath}
\omega_h(g)=\widehat{\widehat{\omega_{h^\dag}}}(g)=\widehat{\omega_{h^\dag}}(\widehat{g})=\dens(\cL)\cdot \omega_{\widecheck{h^\dag}}(\widehat{g})=\dens(\cL)\cdot\omega_{\widecheck h}(\widecheck g)\, ,
\end{displaymath}
where $h^\dag(y)=h(-y)$.
Here the first equation is double transformability of~$\omega_h$, and the second equation is transformability of~$\widehat{\omega_{h^\dag}}$. The statement for $h\in LK(H)$ now follows from the above arguments by double transformability of~$\omega_{\widecheck h}$.
\end{proof}

\subsection{Averages and van Hove sequences}

With respect to Fourier analysis, characteristic functions of van Hove sequences behave similarly to rectangular functions of increasing width in~$\mathbb R$.

\begin{lemma}\cite[Lem,~3.14]{RS15}\label{lem:rect}
Let~$G$ be a $\sigma$-compact LCA group and let $(A_n)_{n\in \mathbb N}$ be a van Hove sequence in~$G$. We then have pointwise
\begin{displaymath}
\lim_{n\to\infty} \frac{1}{\theta_G(A_n)} \widehat{\ifu {A_n}} \to \ifu {\{1\}} \,,
\end{displaymath}
where $1\in \widehat G$ denotes the identity character.
\qed
\end{lemma}

\section{Stable interpolation by model sets}\label{sec:sims}

\subsection{An interpolation upper bound}

As is well known, the interpolation upper bound holds for arbitrary uniformly discrete point sets. Here we give a short direct proof based on Dirac nets and Fourier analytic arguments. These will later reappear in the proof of the Duality Theorem~\ref{thm:main}. Recall that a \emph{Dirac net} in a general LCA group~$G$ is a net $(v_j)_{j\in J}$ in~$G$ such that $v_j\in C_c(G)$,  $v_j\ge 0$ and  $\int_{G} v_j\dd\theta_{G}=1$ for all~$j$. Furthermore the support of~$v_j$ shrinks, i.\,e., for every zero neighborhood $U\subset G$ there exists $j_0\in J$ such for all $j\ge j_0$ we have $\supp(v_j)\subset U$.   For the construction of such a net see e.\,g.~\cite[Lem.~1.6.5]{DE}. If the LCA group~$G$ is metrizable, we can work with Dirac sequences $(v_n)_{n\in \mathbb N}$ instead of Dirac nets.

\begin{lemma}[Upper bound for interpolation from uniformly discrete sets] \label{lem:iub}
Let~$G$ be any LCA group. Let $\Lambda\subset G$ be any uniformly discrete point set in~$G$. Consider any measurable relatively compact $K\subset  \widehat G$.
Then there exists a finite constant $B=B(\Lambda, K)$ such that
  \begin{displaymath}
      \norm{\ifu Kf_\varphi}_{L^2(\widehat G)}^2
    \le B\norm\varphi_{\ell^2(\Lambda)}^2
  \end{displaymath}
 for all finitely supported $\varphi\from\Lambda\to\C$.
\end{lemma}

Recall that $f_\varphi\in C(\widehat G)$ denotes the discrete Fourier transform of~$\varphi$, which is defined by
  $f_\varphi(\chi)=\sum_{\lambda\in\Lambda}\varphi(\lambda)\overline{\chi(\lambda)}$.

\begin{proof}
Take any Dirac net $(v_j)_{j\in J}$ in~$G$ of positive definite functions, compare \cite[Lem.~3.4.1]{DE}. Fix $\varepsilon\in (0,1)$ and choose $j=j(\Lambda, K, \varepsilon)$ sufficiently large such that the following two properties are satisfied.
\begin{itemize}
\item[(a)] For any $\lambda\in \Lambda$, the set $\supp(v_j)+\lambda$ contains exactly one point of~$\Lambda$. This is possible as~$\Lambda$ is uniformly discrete. 
\item[(b)] We have $\widehat{v_j}\ge \varepsilon >0$ on~$K$. This is possible as $\widehat{v_j}$ is non-negative for every~$j$ and since  $(\widehat{v_j})_{j\in J}$ converges locally uniformly to one by \cite[Lem.~3.4.5]{DE}. 
\end{itemize}
Now take any $\varphi\in \ell^2(\Lambda)$ of finite support and define $f_j\in L^2(G)\cap C_c(G)$ by
\begin{displaymath}
f_j(x)=\frac{1}{\norm{v_j}_{L^2(G)}} \sum_{\lambda\in \Lambda} v_j(x-\lambda) \varphi(\lambda) \,.
\end{displaymath}
Note that by the support condition on~$v_j$ we have $\norm{\varphi}_{\ell^2(\Lambda)}^2=\norm{f_j}_{L^2(G)}^2=\norm{\widehat {f_j}}_{L^2(\widehat G)}$. A straightforward calculation shows $\widehat{f_j}\in L^1(\widehat G)\cap L^2(\widehat G)$, where
\begin{displaymath}
\widehat{f_j}=\frac{\widehat{v_j}}{\norm{v_j}_{L^2(G)}}\cdot f_\varphi\ .
\end{displaymath}
We can thus conclude
\begin{displaymath}
\norm{\widehat{v_j}\cdot f_\varphi}_{L^2(\widehat G)}=\norm{v_j}_{L^2(G)}\cdot \norm{\varphi}_{\ell^2(\Lambda)}\,.
\end{displaymath}
Since $0 < \varepsilon\le \widehat{v_j}$ on~$K$, we have
\begin{displaymath}
\varepsilon \cdot \norm{\ifu K\cdot f_\varphi}_{L^2(\widehat G)}\le \norm{v_j}_{L^2(G)}\cdot \norm{\varphi}_{\ell^2(\Lambda)}
\end{displaymath}
for every $\varphi\in \ell^2(\Lambda)$ of finite support.
Thus $B=\norm[L^2(G)]{v_j}^2/\varepsilon^2$ gives an upper bound for stable interpolation.
\end{proof}

\subsection{An interpolation lower bound}

\begin{prop}[Interpolation lower bound]\label{lem:ilbms} Consider any bare cut-and-project scheme $(G,H,\cL)$, where~$G$ is metrizable. Take any relatively compact $W\subset H$ and consider the projection set $\Lambda_W\subset G$. Then there exist a compact set $K=K(W)\subset \widehat G$ and a positive constant $A=A(W,K)$ such that
  \begin{displaymath}
    A\norm[\ell^2(\Lambda_W)]\varphi^2 \le  \norm[L^2(\widehat G)]{\ifu Kf_\varphi}^2
  \end{displaymath}
 for all finitely supported $\varphi\from\Lambda_W\to\C$.
\end{prop}

\begin{remark}\label{rem:int}
The following proof shows how to obtain suitable $K\subset \widehat G$. Here we summarize the construction. Fix any compact unit neighborhood $V\subset \widehat G$ and take non-negative and positive definite  $\varphi_V\in C_c(\widehat G)$ such that $\supp(\varphi_V)\subset V$ and $\int_{\widehat G}\varphi_V\dd\theta_{\widehat G}=1$.  For any relatively compact $K\subset\widehat G$ the function $g_K=\ifu {K\setminus \partial^V K}*\varphi_V$ satisfies $g_K\in KL(\widehat G)$ by H\"older's inequality and  $\ifu K\ge g_K \ge 0$. Here we denote by $f*g$ the convolution of~$f$ and~$g$. Now fix a zero neighborhood $B\subset G$ such that $(\Lambda_W-\Lambda_W)\cap B=\{0\}$ and choose relatively compact $K\subset \widehat G$ such that
\begin{displaymath}
\omega_{\ifu {W-W}}\left(\ifu {B^c}\cdot \abs*{\frac{\widehat{g_K}}{\widehat{g_K}(0)}}\right)<1\, .
\end{displaymath}
The positive constant~$A$ in the proposition is then given by
\begin{displaymath}
A=\left( \widehat{g_K}(0) -
\omega_{\ifu {W-W}}\left( \ifu {B^c} \cdot\abs*{\widehat{g_K}}\right) \right) \,.
\end{displaymath}
For example,~$K$ may be taken as a member of any van Hove sequence $(K_n)_n$ of sufficiently large index. To see this, we may argue as in the proof of \cite[Prop.~3.14]{RS15}. Note first that van Hove sequences exist in~$\widehat G$ as~$\widehat G$ is metrizable. Furthermore $(K_n\setminus \partial^V K_n)_{n\in\mathbb N}$ is also a van Hove sequence, which can be seen by straightforward estimates. We can then use Lemma~\ref{lem:rect} to infer $\frac{\widehat{g_{K_n}}}{\widehat{g_{K_n}}(0)}\to \ifu{\{0\}}$ as $n\to \infty$. Moreover the latter functions are dominated by the non-negative continuous function $\widehat{\varphi_V}\in LK(G)$. In fact we have $\widehat{\varphi_V}\in L^1(\omega_{\ifu {W-W}})$, since $\ifu {W-W}$ may be majorised by suitable $h\in KL(H)$, and since we have $\widehat{\varphi_V} \in L^1(\omega_h)$ by Lemma~\ref{lem:FTWMS}.  Hence the claim follows from dominated convergence. 
\end{remark}

\begin{proof}[Proof of Proposition~\ref{lem:ilbms}]
Note $\Lambda_W-\Lambda_W\subset \Lambda_{W-W}$, where the latter set is uniformly discrete as~$W$ is relatively compact. We may thus choose any zero neighborhood $B\subset G$ such that $(\Lambda_W-\Lambda_W)\cap B=\{0\}$.
Take any $g_K\in KL(\widehat G)$ such that $\ifu K \ge g_K\ge 0$, compare Remark~\ref{rem:int}.
We can then write
\begin{displaymath}
\begin{split}
\norm[L^2(\widehat G)]{\sqrt{g_K}\cdot f_\varphi}^2&=
\sum_{\lambda, \lambda'\in\Lambda_W} \varphi(\lambda) \overline{\varphi(\lambda')} \int g_K(\chi) \overline{\chi(\lambda-\lambda')} \dd\theta_{\widehat G}(\chi)=
\sum_{\lambda, \lambda'\in\Lambda_W} \varphi(\lambda) \overline{\varphi(\lambda')} \, \widehat{g_K}(\lambda-\lambda') \\
&= \norm[\ell^2(\Lambda_W)]{\varphi}^2 \cdot \norm[L^1(\widehat G)]{g_K}
+\sum_{\lambda, \lambda'\in\Lambda_W} \varphi(\lambda) \overline{\varphi(\lambda')} \, (\widehat{g_K}\cdot \ifu {B^c})(\lambda-\lambda')\,.
\end{split}
\end{displaymath}
%
The latter sum can be simplified as follows. We first note that for non-negative $h:G\to \mathbb R$ we have uniformly  in $\lambda\in\Lambda_W$ the estimate
\begin{displaymath}
\sum_{\lambda'\in\Lambda_W} h(\lambda-\lambda') \le
\sum_{\lambda'' \in\Lambda_W-\Lambda_W} h(\lambda'')\le
\sum_{\lambda''\in \Lambda_{W - W}} h(\lambda'')\,.
\end{displaymath}
We can now estimate the above sum using an argument as in the proof of Young's inequality \cite[Thm.~20.18]{HeRo} where $p=q=2$ and $r=1$ by
\begin{displaymath}
\begin{split}
  \abs*{ \sum_{\lambda, \lambda'\in\Lambda_W} \varphi(\lambda)\, (\widehat{g_K}\cdot \ifu {B^c})(\lambda-\lambda')\, \overline{\varphi(\lambda')} } &
  \le \norm[\ell^2(\Lambda_W)]{\varphi} \cdot \sum_{\lambda\in \Lambda_{W - W}} (\abs{\widehat{g_K}}\cdot \ifu {B^c})(\lambda) \cdot  \norm[\ell^2(\Lambda_W)]{\varphi}\\
&\le \norm[\ell^2(\Lambda_W)]{\varphi}^2 \cdot \omega_{\ifu {W-W}}(\abs{\widehat{g_K}}\cdot \ifu {B^c})\,.
\end{split}
\end{displaymath}
If~$\cL$ projects injectively to~$G$, then the last inequality is an equality. Collecting the above estimates, we arrive at
\begin{displaymath}
\norm[L^2(\widehat G)]{\ifu K\cdot f_\varphi}^2\ge \norm[L^2(\widehat G)]{\sqrt{g_K}\cdot f_\varphi}^2\ge \norm[\ell^2(\Lambda_W)]{\varphi}^2 \left( \norm[L^1(\widehat G)]{g_K} - \omega_{\ifu {W-W}}(\abs{\widehat{g_K}}\cdot \ifu {B^c})\right)\,.
\end{displaymath}
The term in brackets will be positive if~$K$ is a member of any van Hove sequence of sufficiently large index, compare Remark~\ref{rem:int}. 
\end{proof}

\section{Stable sampling by model sets}\label{sec:samp}

\subsection{A sampling upper bound}\label{sec:ub}

Since~$\PW_K$ is closed, Hilbert space techniques are available. One can show that the existence of an upper bound for stable sampling is equivalent to the existence of an upper bound for stable interpolation, see \cite[p.~129]{Y01}. A direct proof of an upper bound is, for general uniformly discrete sets, given in \cite[Lem.~2]{GKS08}. In the following, we give an alternative argument for weak model sets based on the Poisson Summation Formula. It yields an explicit constant appearing in the upper bound. This serves as a preparation for the lower bound, where the following argument will be adapted.

\begin{lemma}[Explicit sampling upper bound for weak model sets]\label{lem:ubs}
Consider any bare cut-and-project scheme $(G,H,\cL)$. Take any relatively compact $W\subset H$ and any relatively compact $K\subset \widehat G$. Define the finite constant $B=B(W,K)$ by $B=\dens(\cL)\cdot  \omega_{\abs{\widecheck h}}(K-K)$, where $h\in KL(H)$ is any function satisfying $h \ge \ifu W$. We then have
\begin{displaymath}
\sum_{\lambda\in \Lambda_W} \abs{f(\lambda)}^2 \le B \cdot  \norm{f}_{L^2(G)}^2
\end{displaymath}
for all $f\in \PW_K$.
\end{lemma}

\begin{remark}\label{rem:gr}
A suitable function $h\ge \ifu W$ may be constructed as follows. Take any  compact zero neighborhood~$U$ in~$H$ and take non-negative $\varphi_U\in L^2(H)$  such that $\supp(\varphi_U)\subset U$ and $\int_H\varphi_U\dd\theta_H=1$. Then $h=\ifu {W-U} * \varphi_U$ satisfies $h\ge \ifu W$ by construction, and we have $h\in KL(H)$ by \cite[Lem.~3.7]{RS15}.
\end{remark}

\begin{proof}[Proof of Lemma~\ref{lem:ubs}]
Fix any relatively compact window $W\subset H$ and $K\subset \widehat G$. Take any $h\in KL(H)$ such that $h\ge \ifu W$, compare Remark~\ref{rem:gr}.  Consider any $f\in \PW_K$. Then $g=\abs f^2\in LK(G)$, and we have $\widecheck{g}=\widecheck{f\cdot \overline{f}}=\widecheck{f}*\widecheck{\overline{f}}=
\widecheck{f}*\widetilde{\widecheck{f}}$. Here we used the notation $\widetilde f(x)=\overline{f(-x)}$.  Using the Poisson Summation Formula Lemma~\ref{lem:FTWMS}, we can write
\begin{displaymath}
\sum_{\lambda\in \Lambda_W} g(\lambda) \le \omega_{\ifu W}(g)\le \omega_h(g)=\dens(\cL)\cdot \omega_{\widecheck h}(\widecheck g)
=\dens(\cL)\cdot \omega_{\widecheck h}(\widecheck f * \widetilde{\widecheck f})\,.
\end{displaymath}
This first inequality is an equality if~$\cL$ projects injectively to~$G$.
Now we can apply a standard estimate, using that $\omega_{\widecheck h}$ is a complex Radon measure and that $\supp(\widecheck f * \widetilde{\widecheck f})\subset K-K$. We obtain
\begin{displaymath}
\dens(\cL)\cdot \omega_{\widecheck h}(\widecheck f * \widetilde{\widecheck f}) \le B \cdot \norm{\widecheck f * \widetilde{\widecheck f}}_\infty \le  B \cdot \norm{\widecheck f }_2 \cdot  \norm{\widecheck f }_2 \,,
\end{displaymath}
where the second estimate relies on Young's inequality.
The finite constant $B=B(K,W)$ can be chosen independently of the continuous function $f\in \PW_K$ by
$B=\dens(\cL)\cdot\omega_{\abs{\widecheck h}}(K-K)$.
\end{proof}


%
%


\subsection{A sampling lower bound}

\begin{prop}[Sampling lower bound]\label{lem:slb}
Consider any bare cut-and-project scheme $(G,H,\cL)$. Assume that~$H$ is $\sigma$-compact and that~$\cL$ projects both injectively and densely to~$G$. Take any relatively compact $K\subset \widehat G$. Then there exist compact  $W=W(K)\subset H$ and a positive constant $A=A(K,W)$ such that
\begin{displaymath}
A \cdot  \norm{f}_{L^2(G)}^2\le \sum_{\lambda\in \Lambda_W} \abs{f(\lambda)}^2 
\end{displaymath}
for any $f\in \PW_K$.
\end{prop}

\begin{remark}\label{rem:samplow}
The following proof also shows how to obtain suitable  windows~$W$. Here we summarize the construction, which is analogous to that in Remark~\ref{rem:int}. Fix any compact zero neighborhood $U\subset H$ and take non-negative and positive definite $\varphi_U\in C_c(H)$ such that $\supp(\varphi_U)\subset U$ and $\int_H\varphi_U\dd\theta_H=1$. 
For given compact $W\subset H$ define $h_W=\ifu {W\setminus \partial^U W}*\varphi_U$. Then $h_W\in KL(H)$ satisfies $h_W\le \ifu W$ and $\widecheck{h_W}(1)=\theta_H(W\setminus \partial^U W)$. 

 Choose some compact zero neighborhood $B\subset \widehat H$ containing only the origin of the uniformly discrete projection set $_{K-K} \Lambda\subset \widehat H$. 
Now take compact $W\subset H$  such that
\begin{displaymath}
_{\ifu {K-K}}\omega\left( \ifu {B^c} \cdot \abs*{ \frac{\widecheck{h_W}}{\widecheck{h_W}(1)}}\right) < 1\,.
\end{displaymath}
Note that~$H$ admits a Hove sequence $(W_n)_{n\in\mathbb N}$ since~$H$ is $\sigma$-compact. One may choose $W=W_n$ for~$n$ sufficiently large, compare the arguments in Remark~\ref{rem:int}. In that case, the positive constant~$A$ can be chosen as
\begin{displaymath}
A=\dens(\cL)\cdot \left( \widecheck{h_W}(1) -
{_{\ifu {K-K}}}\omega\left( \ifu {B^c} \cdot \abs{\widecheck {h_W} }\right) \right) \,.
\end{displaymath}
\end{remark}


\begin{proof}[Proof of Proposition~\ref{lem:slb}]
Define $h_W\in KL(H)$ as in Remark~\ref{rem:samplow}.
Take $f\in\PW_K$ and note $g=\abs f^2\in LK(G)$.
Using the Poisson Summation Formula Lemma~\ref{lem:FTWMS}, we can write
\begin{displaymath}
\sum_{\lambda\in \Lambda_W} g(\lambda) = \omega_{\ifu W}(g) \ge \omega_{h_W}(g)=\dens(\cL)\cdot \omega_{\widecheck{h_W}}(\widecheck f * \widetilde{\widecheck f}) \,.
\end{displaymath}
In the first equation, we used that~$\cL$ projects injectively to~$G$. Note that the projection set $_{K-K}{\Lambda}\subset \widehat H$ is uniformly discrete as $K-K$ is relatively compact.
 Take any relatively compact unit neighborhood $B\subset \widehat H$ containing only the unit in $_{K-K}{\Lambda}$.
Noting $(\widecheck f * \widetilde{\widecheck f})(1)=\norm{f}_2^2$ and $\widecheck{h_W}(e)=\theta_H(W\setminus \partial^U W)$, we can write
\begin{equation}\label{eq:decomp}
\omega_{\widecheck{h_W}}(\widecheck f * \widetilde{\widecheck f})=
\sum_{(\chi,\eta)\in\cL_0} \left(\widecheck f * \widetilde{\widecheck f}\right)(\chi)\cdot \widecheck{h_W}(\eta)=\norm{f}_2^2\cdot \theta_H(W\setminus \partial^U W) + _{\widecheck f * \widetilde{\widecheck f}}\omega(\ifu {B^c}\widecheck{h_W})\,.
\end{equation}
In the second equality, we used that~$\cL_0$ projects injectively to~$\widehat H$, which is equivalent to $\pi^G(\cL)$ being dense in~$G$.
%
A standard estimate on the second term on the rhs of Eq.~\eqref{eq:decomp}, which uses $\supp(\widecheck f * \widetilde{\widecheck f})\subset K-K$, yields
\begin{displaymath}
\abs{{_{\widecheck f * \widetilde{\widecheck f}}}\omega(\ifu {B^c}\widecheck{h_W})}
\le _{\abs{\widecheck f * \widetilde{\widecheck f}}}\omega(\ifu {B^c}\abs{\widecheck{h_W}})
\le {_{\ifu {K-K}}}\omega(\ifu {B^c}\abs{\widecheck{h_W}})\cdot  \norm{\widecheck f * \widetilde{\widecheck f}}_\infty\le  {_{\ifu {K-K}}}\omega(\ifu {B^c}\abs{\widecheck{h_W}})\cdot \norm{f}_2^2 \,,
\end{displaymath}
where we used Young's inequality \cite[Thm.~20.18]{HeRo} for the last estimate. Summarizing the above estimates, we get
\begin{displaymath}
\sum_{\lambda\in \Lambda_W} \abs{f(\lambda)}^2 \ge \left( \theta_H(W\setminus \partial^U W) - {_{\ifu {K-K}}}\omega(\ifu {B^c}\abs{\widecheck{h_W}})\right) \norm{f}_2^2 \,.
\end{displaymath}
%
The term in brackets will be positive if~$W$ is a member of any van Hove sequence of sufficiently large index, compare Remark~\ref{rem:samplow}. Here we use that~$H$ is $\sigma$-compact.
\end{proof}

\section{Lower bounds and critical density}\label{sec:cdlb}

We analyze the above method for sampling using model sets in $G=\mathbb R$ which are simple, i.\,e., which satisfy $H=\mathbb R$ and $W\subset H$ an interval. We will see that the method used above does not give sharp results.

\subsection{Sampling upper bound}

We give an explicit upper bound for the statement in Lemma~\ref{lem:ubs}.
Choose centered intervals  $U=[-u,u]$  and $W=[-w,w]\subset H$, where $u,w>0$. Define $\varphi_U=\frac{1}{2u}\ifu U$ and $h= \ifu {W-U} *\varphi_U$. Then $h\in KL(H)$ is piecewise linear, and we have $h=1$ on~$W$ and $h=0$ on $(W+U)^c$. The transform of~$h$ is given by
\begin{displaymath}
\widehat h(k)= \widehat{\ifu {W-U}}(k) \cdot \widehat {\varphi_U}(k)
= \frac{\sin(2\pi (w+u)k)}{\pi k} \cdot \frac{\sin(2\pi u k)}{2\pi u k}\,.
\end{displaymath}
We have $\widehat h(0)=2(w+u)$, and we may later use the standard estimate
$\abs{\widehat h(k)}\le (2\pi^2 u k^2)^{-1}$.
Due to uniform discreteness of the projection set $_{K-K} \Lambda$, we may assume that $b>0$ is sufficiently small such that any translation of  $B=[-b,b)\subset \widehat H$ contains at most one point of~$_{K-K}  \Lambda$. Using an interval partition of~$\widehat H$ by shifted copies of~$B$,  we can estimate
\begin{displaymath}
\begin{split}
\omega_{\abs{\widecheck h}}(\ifu {B^c}\ifu {K-K})&= \sum_{\ell\in\mathbb N} {_{\ifu {K-K}} \omega}  (\abs{\widecheck h} \cdot (\ifu {B+2\ell b}+\ifu {B-2\ell b}))\le \sum_{\ell\in\mathbb N} \sup\{\abs{\widecheck h(k)}: k\in (B+2\ell b) \cup (B-2\ell b)\}\\
&\le 2\sum_{\ell\in\mathbb N} \sup\left\{\frac{1}{2\pi^2 u k^2}: k\in [b(2\ell-1), b(2\ell+1))\right\}
\le \sum_{\ell\in\mathbb N} \frac{1}{\pi^2 u b^2 (2\ell-1)^2}=\frac{1}{8ub^2}\,.
\end{split}
\end{displaymath}
Here we used $\sum_{\ell\in\mathbb N} \frac{1}{(2\ell-1)^{2}}=\pi^2/8$. Hence we may require the finite stable sampling upper bound to satisfy
\begin{displaymath}
C=\dens(\cL)\cdot\omega_{\abs{\widecheck h}}(\ifu {K-K})\le \dens(\cL)\cdot\left(2(w+u)+\frac{1}{8ub^2}\right)\,.
\end{displaymath}
At $u=1/(4b)$, the latter expression takes its minimum $\dens(\cL)\cdot(2w+1/b)$.

\subsection{Sampling lower bound} We may proceed similarly for the upper bound case. Choose centered intervals $W=[-w,w]\subset H$ and $U=[-u,u]$ where $0<u<w$. Define $\varphi_U=\frac{1}{2u}\ifu U$ and $h_W= \ifu {W\setminus \partial^UW} *\varphi_U$. Then $h_W\in KL(H)$ is piecewise linear, and we have $h_W=1$ on $[-(w-u),w-u]$ and $h_W=0$ on~$W^c$.
%
%
We have $\widehat{h_W}(0)=2(w-u)$, and as above we have the standard estimate $\abs{\widehat{h_W}(k)}\le (2\pi^2 u k^2)^{-1}$.
Due to uniform discreteness of the projection set~$_{K-K} \Lambda$, we may assume that $b>0$ is sufficiently small such that any translate of $B=[-b,b)\subset \widehat H$ contains at most one point of~$_{K-K} \Lambda$. Using an interval partition of~$\widehat H$ by shifted copies of~$B$,  we can estimate as above
\begin{displaymath}
\begin{split}
\omega_{\abs{\widecheck {h_W}}}(\ifu {B^c}\ifu {K-K})\le \frac{1}{8ub^2}\,.
\end{split}
\end{displaymath}
Using the Poisson Summation Formula, we thus get
\begin{displaymath}
\begin{split}
\frac{\omega_{\ifu W}(g)}{\dens(\cL)} &\ge \frac{\omega_{h_W}(g)}{\dens(\cL)} = \omega_{\widecheck g}(\ifu B\cdot \widecheck {h_W}) +
\omega_{\widecheck g}(\ifu {B^c}\cdot\widecheck {h_W}) \ge
\norm g_1 \cdot \norm{h_W}_1 - \abs{\omega_{\widecheck g}(\ifu {B^c}\cdot\widecheck {h_W}) }\\ &\ge \norm{g}_1 \left(2(w-u)-\frac{1}{8ub^2}\right) \,.
\end{split}
\end{displaymath}
A necessary condition for the rhs to be positive can be explicitly calculated as $2w b> 1$. Let us rephrase this result in terms of the density of the regular model set~$\Lambda_W$, which is given by  $\dens(\Lambda_W)=\dens(\cL)\cdot \theta_H(W)$, compare Fact~\ref{thm:df2}. Assume that~$K$ is a centered interval and that~$_{K-K} \Lambda$ has maximal~$b$, i.\,e., $(2b)^{-1}=\dens(_{K-K} \Lambda)=\dens(\cL_0) \cdot \theta_{\widehat G}(K-K)$. We then get
\begin{displaymath}
2w b> 1 \qquad \Longleftrightarrow \qquad \dens(\Lambda_W)>4\theta_{\widehat G}(K) \, .
\end{displaymath}
This is away from Landau's necessity condition $\dens(\Lambda_W)\ge \theta_{\widehat G}(K)$ by a factor of 4.

\section{Density formulae for weighted model combs}\label{sec:densf}

\subsection{Point densities}

Regular model sets have a uniform point density. This result can be seen as a consequence of the Poisson Summation Formula, an approach employed by Meyer  \cite{me70},  \cite[Sec.~V.7.3]{me72}, \cite[Prop.~5.1]{mm10}. The density formula for a weighted model comb~$\omega_h$  generalizes uniform point density of the underlying model set $\supp(\omega_h)$, see e.g\,.~\cite[Thm.~4.14]{RS17}.

%
%

\begin{fact}[Density formula for Riemann integrable weight functions]\label{thm:df}
Let $(G,H,\mathcal L)$ be a bare cut-and-project scheme. Assume that~$G$ is $\sigma$-compact and that~$\cL$ projects densely to~$H$.
If $h : H \to \mathbb C$ is Riemann integrable, then for every van Hove sequence sequence
$(A_n)_{n\in\mathbb N}$ in~$G$ the density formula holds, i.\,e., for every $s\in G$ we have
\begin{displaymath}
\lim_{n\to\infty} \frac{\omega_h(s+A_n)}{\theta_G(A_n)}=\dens(\mathcal L)\cdot \int_H h \dd\theta_H \,.
\end{displaymath}
The convergence is uniform in $s\in G$.
\qed
\end{fact}

The density formula implies that any regular model set $\Lambda_W$ has a uniform point density. Indeed, then $h=\ifu W$ is Riemann integrable. For a general weak model set, only upper and lower estimates can be given. The following result can  be proved by approximation with regular model sets using the density formula, compare \cite[Prop.~3.4]{HR15}.
A version for $G=\mathbb R$ was already used by Meyer  in \cite[Rem.~6.2]{me73}.

\begin{fact}\label{thm:df2}
Let $(G,H,\mathcal L)$ be a bare cut-and-project scheme. Assume that~$G$ is $\sigma$-compact and that~$\cL$ projects densely to~$H$.
For any weak model set~$\Lambda_W$ in $(G,H,\cL)$, any van Hove sequence~$\cA$ in~$G$ and any $t\in H$, we have the estimates
\begin{equation}\label{eq:denswms}
\dens(\cL) \cdot \theta_H(\inn{W})\le \cd^-_\cA(\Lambda_{W+t}) \le  \cd^+_\cA(\Lambda_{W+t}) \le \dens(\cL) \cdot \theta_H(\cl{W})\,.
\end{equation}
In particular if~$W$ is Riemann measurable, then $\cd(\Lambda_{W+t})=\dens(\cL) \cdot \theta_H(W)$.
\qed
\end{fact}

\subsection{Density formula for continuous test functions}

The density formula Fact~\ref{thm:df} expresses that the weighted model comb~$\omega_h$ has a uniform density.
Indeed we have $\omega_h(g_n)\to \dens(\cL)\cdot \theta_H(h)$ where $g_n=\ifu {s+A_n}/\theta_G(A_n)$.
Below we need to compute the limit in the density formula on a sequence of continuous test functions $(g_n)_n$ of unbounded support. If~$G$ is $\sigma$-compact, then such sequences may be obtained using a Dirac sequence $(v_n)_n$ in~$\widehat G$, as the sequence $(\widehat v_n)_n$ then converges locally uniformly to 1, see \cite[Lem.~3.4.5]{DE}. We have the following result.

\begin{prop}[Density formula with continuous averaging functions]\label{denssm}
Let $(G,H,\cL)$ be a bare cut-and-project scheme. Assume that~$G$ is $\sigma$-compact and that~$\cL$ projects densely to~$H$. Take a Dirac sequence $(v_n)_{n\in\mathbb N}$ in~$\widehat G$ and define $g_n\in L^1(G)$ by
\begin{displaymath}
g_n=\frac{\abs{\widehat{v_n}}^2}{\norm[L^2(G)]{\widehat{v_n}}^2}\,.
\end{displaymath}
Then for any Riemann integrable $h: H\to\mathbb C$ and for every $s\in G$ we have
\begin{displaymath}
\lim_{n\to\infty} \omega_h(\delta_s*g_n)=\dens(\mathcal L)\cdot \int_H h \dd\theta_H\,.
\end{displaymath}
The convergence is uniform in $s\in G$.
\end{prop}

\medskip

The proof of Proposition \ref{denssm} is based on a version of the Fourier-Bohr coefficient formula, compare \cite[Prop.~3.12]{RS15}, for continuous averaging functions.

\begin{prop}[Fourier-Bohr coefficients with continuous averaging functions]\label{prop:FBsmooth}
Let~$G$ be a $\sigma$-compact LCA group. Take any Dirac sequence $(v_n)_{n\in\mathbb N}$ in~$\widehat G$ and define $g_n\in L^1(G)$ by
\begin{displaymath}
g_n=\frac{\abs{\widehat{v_n}}^2}{\norm[L^2(G)]{\widehat{v_n}}^2}\,.
\end{displaymath}
Let $\mu\in\mathcal M^\infty(G)$ be twice transformable and consider $\chi\in\widehat G$. We then have for every $s\in G$
\begin{displaymath}
\widehat \mu(\{\chi\})=\lim_{n\to\infty} \int_{G} g_n(s+x) \overline{\chi}(x) \dd \mu(x) \,.
\end{displaymath}
The convergence is uniform in $s\in G$.
\end{prop}

\begin{proof}[Proof of Proposition~\ref{prop:FBsmooth}]
We prove the proposition for $\chi=1$, where~$1$ denotes the trivial character. The general case then follows from $\widehat \mu(\{\chi\})=(\delta_{\chi^{-1}}*\widehat \mu)(\{1\})$ and $(\delta_{\chi^{-1}} *\widehat\mu)=\widehat {\overline \chi \mu}$, where $*$ denotes measure convolution. Recall that~$\widehat G$ is metrizable since~$G$ is $\sigma$-compact. Hence there exists a Dirac sequence $(v_n)_{n\in \mathbb N}$ in $\widehat G$.  Note first that
\begin{displaymath}
\widecheck{g_n}=\frac{v_n*\widetilde{v_n}}{\norm[L^2(\widehat G)]{v_n}^2} \in C_c(\widehat G) \, .
\end{displaymath}
%
We also have $\widecheck{g_n}(1)=1$ and $0\le \widecheck{g_n} \le 1$ by Young's inequality.
The function $\widecheck{g_n}$ is positive definite and compactly supported, and its Fourier transform~$g_n$ is integrable.
As $\widehat \mu\in\mathcal M(\widehat G)$ is transformable, we thus have
\begin{equation}\label{eqn:PSFFB}
\widehat{\mu}\left( e_s\cdot \widecheck{g_n}\right)=\widecheck{\widehat{\mu}}(\delta_{-s}*g_n)=\mu( \delta_{-s}*g_n)
\end{equation}
by the inversion theorem for transformable measures \cite[Thm.~3.4]{ARMA1} and by definition of the Fourier transform \cite[Ch.~2]{ARMA1}. In Eqn.~\eqref{eqn:PSFFB}, recall that~$e_s$ is the exponential $e_s(\chi)=\chi(s)$.
Let us consider the limit $n\to\infty$ on the lhs.  We first note $\widecheck{g_n}(\chi)\to \ifu {\{1\}}$ pointwise. Indeed, this is obvious for $\chi=1$, and for $\chi\ne 1$ we choose any compact unit neighborhood~$K$ in~$\widehat G$ such that $\chi\notin K$. Then $\supp(\widecheck{g_n})\subset K$ for finally all~$n$, such that $\int_G g_n(x)\chi(x) \dd\theta(x)=\widecheck{g_n}(\chi)=0$ for finally all~$n$. 
%
%
As $\abs{e_s\cdot\widecheck{g_n}-\ifu {\{1\}}}=\abs{\widecheck{g_n}-\ifu {\{1\}}}\le\abs{\widecheck{g_n}}\le 1$,
we can write
\begin{displaymath}
\abs{\widehat\mu (e_s\cdot \widecheck{g_n})-\widehat\mu (\ifu {\{1\}})}\le
\abs{\widehat \mu}(\abs{s\cdot \widecheck{g_n}-\ifu {\{1\}}})=
\abs{\widehat \mu}(\abs{ \widecheck{g_n}-\ifu {\{1\}}})\le
\abs{\widehat \mu}(\abs{\widecheck{g_n}})\le
\abs{\widehat\mu}(K) \cdot \norm{\widecheck{g_n}}_\infty\le \abs{\widehat\mu}(K)
\end{displaymath}
for finally all~$n$. We can thus apply Lebesgue's dominated convergence theorem to obtain
$\widehat\mu (e_s\cdot \widecheck{g_n})\to \widehat\mu (\ifu {\{1\}})=\widehat \mu(\{1\})$ as $n\to\infty$ uniformly in $s\in G$.
\end{proof}

\begin{proof}[Proof of Proposition \ref{denssm}]
Abbreviate $c=\dens(\mathcal L)>0$, take  $\varphi\in C_c(H)$ and consider $h=\varphi*\widetilde \varphi\in C_c(H)$. Then~$\omega_h$ is twice transformable by \cite[Thm.~4.12]{RS15}.  We can now argue
\begin{displaymath}
c\cdot\widecheck{h}(1)=c\cdot \omega_{\widecheck{h}}(\{ 1\})=\widehat{\omega_h}(\{1\}) =\lim_{n\to\infty}  \omega_h(\delta_s*g_n)
\end{displaymath}
uniformly in $s\in G$.
In the first equation we used that~$\omega_{\widecheck h}$ is a measure and that~$\cL_0$ projects injectively to~$\widehat G$, which is equivalent to~$\cL$ projecting densely to~$H$ by Pontryagin duality. In the second equation, we used the generalized Poisson Summation Formula $\widehat{\omega_h}=c\cdot \omega_{\widecheck{h}}$, see Lemma~\ref{lem:FTWMS}, and for the third equation we used Proposition~\ref{prop:FBsmooth} as $\omega_h\in\mathcal M^\infty(G)$ is twice transformable.

The general case follows by a standard approximation argument. Consider without loss of generality any Riemann integrable $h:H\to \mathbb R$ and let $\varepsilon >0$.
Recall from \cite{ARMA1} that $K_2(H)=\lin\{\varphi*\widetilde \varphi: \varphi\in C_c(H)\}$ is dense in $C_c(H)$ with respect to the supremum norm.  Since~$h$ is Riemann integrable, there exist $f_1,f_2 \in K_2(H)$ such that $f_1 \leq h \leq f_2 $ and $\int (f_2 -f_1)\dd\theta_H  \le \frac{\varepsilon}{2c}$.
By the above result
there exists~$N$ such that for all $n \ge N$ and $i\in\{1,2\}$ we have
\begin{displaymath}
\abs*{ \omega_{f_i}(\delta_s*g_n)-c \int f_i \dd\theta_H }
\le \frac{\varepsilon}{2}\,,
\end{displaymath}
uniformly in $t\in G$. Thus, as $\omega_{f_1} \leq \omega_h \leq \omega_{f_2}$, for all $n \geq N$ we have
\begin{eqnarray*}
\begin{aligned}
&c \int h  \dd\theta_H - \omega_{h}(\delta_s*g_n)  \leq
c \int f_2  \dd\theta_H - \omega_{f_1}(\delta_s*g_n)   \leq
\frac{\varepsilon}{2}+c \int f_1  \dd\theta_H - \omega_{f_1}(\delta_s*g_n) \le \varepsilon \,,
\end{aligned}
\end{eqnarray*}
and similarly
\begin{eqnarray*}
\begin{aligned}
&c \int h  \dd\theta_H - \omega_{h}(\delta_s*g_n) \ge
-\varepsilon \,,
\end{aligned}
\end{eqnarray*}
both uniformly in $s\in G$. Hence the claim of the proposition follows.
\end{proof}

\section{Sampling and interpolation duality}\label{sec:dualityms}

\subsection{Statement and discussion of the duality theorem}\label{sec:statedis}

Recall that $\CS_K=\{f\in L^1(G)\cap L^2(G) :  \widehat f |_{K^c}=0 \}$ is a subspace of $\PW_K$, which is dense in~$\PW_K$ if~$K$ is Riemann measurable. 
The duality theorem relates, under suitable assumptions, sampling on~$\Lambda_W$ for  $\CS_{K}$  to interpolation from~$_{K} \Lambda$ for~$\PW_{W}$.
We say that~$\Lambda$ is stable sampling for~$\CS_K$ if Eqn.~\eqref{eq:stablesampling} holds for any $f\in \CS_K$.

\begin{theorem}[Sampling-interpolation duality for model sets]\label{thm:main}
	Let $(G,H,\mathcal L)$ be a complete cut-and-project scheme.
	Let $W\subset H$ and $K\subset \widehat G$ both be relatively compact and measurable. Then the following hold.
  \begin{enumerate}
    \item[(i)] Assume that~$G$ is $\sigma$-compact. Take any relatively compact and measurable neighborhood~$K'$ of~$K$.   If $\Lambda_W\subset G$ is stable sampling for $\CS_{K'}\subset L^2(G)$, then $_{K} \Lambda\subset \widehat H$ is stable interpolating for $\PW_{W}\subset L^2(\widehat H)$.
  \item[(ii)] Assume that~$H$ is metrizable. If $_{K} \Lambda \subset \widehat H$ is stable interpolating for $\PW_{W}\subset L^2(\widehat H)$, then $\Lambda_{\cl{W}} \subset G$ is stable sampling for $\CS_K\subset L^2(G)$.
  \end{enumerate}

\end{theorem}

\begin{remark}
In applications, it might be convenient to replace~$(i)$ by the following equivalent statement. It is obtained from~$(i)$ by exchanging the roles of~$G$ and~$H$ and using the dual cut-and-project scheme.
{\it
  \begin{enumerate}
    \item[(i')] Assume that~$H$ is $\sigma$-compact. Take any relatively compact and measurable neighborhood~$W'$ of~$W$.   If $_{K}\Lambda\subset \widehat H$ is stable sampling for $\CS_{W'}\subset L^2(\widehat H)$, then $\Lambda_W\subset G$ is stable interpolating for $\PW_{K}\subset L^2(G)$.
\end{enumerate}
}
\end{remark}

The assumptions in the duality theorem may slightly be relaxed, as is seen by inspecting the proof of the duality theorem.

\begin{remark}[Assumptions in the Duality Theorem]\label{rem:d1}
\leavevmode

\begin{itemize}

\item[(a)] The proof does not use all four projection assumptions in Definition~\ref{def:cps}. The proof of (i) only uses that~$\cL$ projects injectively to~$G$ and densely to~$H$. The proof of (i') only uses that~$\cL$ projects injectively to~$H$ and densely to~$G$. 
The proof of (ii) does not use that~$\cL$ projects densely to~$H$.

\item[(b)] If~$K$ is open, we may take $K'=K$ in (i). For more general~$K$, assume that~$\partial K$ is nowhere dense. This comprises the cases~$K$ compact and~$K$ topologically regular, i.\,e.,   $\cl{\inn{K}}=\cl{K}$. Then, without any projection assumption, nowhere denseness implies that the set of $g\in \widehat G$ such that $(\partial K+g)\cap \pi_{\widehat G}(\cL_0)=\varnothing$ is dense~$G_\delta$ in~$\widehat G$, see the proof of \cite[Prop.~8.3]{KR18} or \cite[Prop.~2.12]{HR15}. Thus by an arbitrarily small shift we may ensure $\partial K\cap \pi_{\widehat G}(\cL_0)=\varnothing$. In that case, we may take $K'=K$ in (i). Analogous reasoning applies to (i').

\item[(c)] If~$W$ is closed, we have $\cl W=W$ in (ii). For more general~$W$, assume that~$\partial W$ is nowhere dense, which comprises the case~$W$ topologically regular. Then, without any projection assumption, by an arbitrarily small shift we may ensure $\partial W\cap \pi_{H}(\cL)=\varnothing$. In that case $\Lambda_{\cl W}=\Lambda_W$ in (ii).

\item[(d)]  Consider the case $G=\mathbb R$ or $G=\mathbb Z$ and assume that $K\subset \widehat G$ is a left-semiclosed multiband set, i.\,e., a union of finitely many pairwise disjoint intervals $[a,b)$. In that situation one may take $K'=K$. This is possible as the structure of~$K$ allows to approach any left interval endpoint in~$K$ from the right. Indeed, in the proof of part (i) of the Duality Theorem~\ref{thm:main}, one may take $\eta_n\in \widehat G \cong\mathbb R$ to be $\eta_n=1/n$  in Eqn.~\eqref{eq3} to see that $K'=K$ is a possible choice.
This has been used in \cite[Lem.~2.1]{KL11} and generalized in \cite{GL16} to the multi-dimensional case.
\end{itemize}
\end{remark}

Note that our Duality Theorem relates~$\PW_{K/W}$ and~$\CS_{W/K}$ in a particular way. This is in contrast to the duality theorems in \cite{KL11, GL16, AACB16}, which all relate~$\PW_{K/W}$ and~$\PW_{W/K}$. In fact those duality theorems are special cases of Theorem~\ref{thm:main}, as can be seen using  Proposition \ref{prop:denseinL2} and the previous remark.

\begin{prop}\label{prop:denseinL2}
Let~$G$ be an LCA group. Let $\Lambda\subset G$ be uniformly discrete and let $K\subset \widehat G$ be relatively compact, measurable  and Riemann measurable. If~$\Lambda$ is stable sampling for~$\CS_K$, then~$\Lambda$ is stable sampling for~$\PW_K$.
\end{prop}

\begin{proof}
Note that the~$\PW_K$ stable sampling upper bound in Eqn.~\eqref{eq:stablesampling} holds for general uniformly discrete sets by \cite[Lem.~2]{GKS08}, see also Lemma~\ref{lem:ubs} for weak model sets. It thus suffices to consider the lower bound.
Assume that~$\Lambda$ is stable sampling for~$\CS_K$ and that~$K$ is Riemann measurable. Fix any $f\in \PW_K$. Take $\varepsilon>0$ and choose $f_\varepsilon\in\CS_K$ such that $\norm{f-f_\varepsilon}_{L^2(G)}\le\varepsilon$, which is possible due to Proposition~\ref{prop:PWdense}. By the triangle inequality we then have
\begin{displaymath}
\abs*{ \norm[L^2(G)]{f} - \norm[L^2(G)]{f_\varepsilon}}
\le \norm[L^2(G)]{f-f_\varepsilon} \le \varepsilon\,.
\end{displaymath}
Moreover, using the upper bound, we can estimate
\begin{displaymath}
\abs*{ \norm[\ell^2(\Lambda)]{f|_\Lambda} - \norm[\ell^2(\Lambda)]{f_\varepsilon|_\Lambda}  }
\le \norm[\ell^2(\Lambda)]{(f-f_\varepsilon)|_\Lambda} \le
 \sqrt{B}\cdot \norm[L^2(G)]{f-f_\varepsilon} \le \sqrt{B}\cdot \varepsilon
\end{displaymath}
This leads to the following estimate for the lower bound:
\begin{displaymath}
\sqrt{A}\cdot \norm[L^2(G)]{f}\le \sqrt{A}\cdot(\norm[L^2(G)]{f_\varepsilon}+\varepsilon)
\le \norm[\ell^2(\Lambda)]{f_\varepsilon|_\Lambda}+\sqrt{A}\cdot\varepsilon
\le \norm[\ell^2(\Lambda)]{f|_\Lambda}+(\sqrt{A}+\sqrt{B})\cdot\varepsilon \,.
\end{displaymath}
As $\varepsilon>0$ was arbitrary, the claim follows.
\end{proof}

\subsection{Proof of the Duality Theorem}

We now give a proof of the Duality Theorem. Since the upper bounds for stable sampling and stable interpolation hold in general, see also Section~\ref{sec:ub} and  Section~\ref{lem:iub} for model sets, we only discuss the lower bounds.

\begin{figure}
  \begin{equation*}
    \begin{CD}
      \CS_{K'} @<<< L^2(K') \\
      @ VVV @AA A \\
      C(H) @< << \ell^2(_K\Lambda)
    \end{CD}
    \qquad\qquad
    \begin{CD}
      f_n @<<< \widehat{f_n} \\
      @VVV @AAA \\
      f_\varphi && \varphi
    \end{CD}
  \end{equation*}
  \caption{Tracing function spaces in the proof of the Duality theorem part (i)}\label{fig1}
\end{figure}

\begin{proof}[Proof of (i)]
  We assume that there is a positive constant~$A_S$ such that
  \begin{equation}\label{eq:stable}
    A_S \cdot \norm{f}_{L^2(G)}^2
    \le \sum_{\lambda\in\Lambda_W}\abs{f(\lambda)}^2
  \end{equation}
for all $f\in\CS_{K'}$, where $K'\supset K$ is any given fixed relatively compact and measurable neighborhood of~$K$.
  We will show that for arbitrary $\varphi\in\ell^2(_K\Lambda)$ with finite support we have
  \begin{equation}\label{eq:interpolation}
    A_S\cdot \norm\varphi_{\ell^2(_K \Lambda)}^2
    \le\dens(\cL)\cdot\int_{W}\abs[\Big]{\sum_{\eta\in {_{K}}{\Lambda}}
      \varphi(\eta)\, \overline{\eta(h)}}^2\dd\theta_H(h) \, .
  \end{equation}
The strategy of the proof consists in associating to $\varphi\in \ell^2(_K\Lambda)$ suitable $f_n\in \CS_{K'}$, which will be defined via their Fourier transform, compare Fig.~\ref{fig1}.

\smallskip

Take any $\varphi\in\ell^2(_K \Lambda)$ of finite support. 
We consider certain $L^2(\widehat G)$ approximations to $\varphi\circ \hat s$, where~$\hat s$ denotes the star map ${\hat s}\from\pi_{\hat G}(\cL_0)\to\pi_{\hat H}(\cL_0)$.
The latter is indeed well-defined as~$\cL_0$ projects injectively to~$\widehat G$, which follows by Pontryagin duality as~$\cL$ projects densely to~$H$.
 Take any sequence $(\eta_n)_n$ of characters in~$\widehat G$ converging to the identity character, see Remark~\ref{rem:d1}~(d). Take a Dirac sequence $(v_n)_{n}$ in~$\widehat G$ such that~$v_n$ is positive definite. This is possible as~$\widehat G$ is metrizable by assumption.  Define
$\widehat f_n\in L^2(\widehat G)\cap C_c(\widehat G)$ to be the finite sum
  \begin{equation}\label{eq3}
    \widehat{f_n}(\gamma)= \frac{1}{\norm{v_n}_{L^2(\widehat G)}} \sum_{\chi\in\pi_{\widehat G}(\cL_0)\cap K} v_n(\gamma\cdot \overline{\chi}\cdot \eta_n)\cdot\varphi(\widehat s(\chi))\,.
  \end{equation}
The latter sum is indeed finite as~$\varphi$ has finite support and as~$K$ is relatively compact.
Note that the support of~$v_n$ is contained in any given unit neighborhood if~$n$ is sufficiently large. Thus $f_n:=\widecheck{\widehat {f_n}}\in L^1(G)$ satisfies $f_n\in \CS_{K'}$ for sufficiently large~$n$. Note that if $\pi_{\widehat G}(\cL_0)\cap \partial K=\varnothing$, we may choose $K'=K$.

\smallskip

We argue that the functions~$f_n$ asymptotically provide the claimed duality.
Note first that for sufficiently large~$n$ we have
\begin{equation}\label{eq1}
  A_S \cdot \norm\varphi_{\ell^2(\pi_{\hat H}(\cL_0))}^2
  =A_S\cdot \norm{\widehat f_n}_{L^2(\widehat G)}^2
  =A_S\cdot \norm{f_n}_{L^2(G)}^2
  \le\sum_{\lambda\in\Lambda_W}\abs{f_n(\lambda)}^2
\end{equation}
by Eqn.~\eqref{eq:stable}.
Let us compute that the rhs is asymptotically bounded by the rhs of Eqn.~\eqref{eq:interpolation}.
We have the explicit expression
\begin{displaymath}
f_n=\frac{\overline{\eta_n}\cdot\widecheck{v_n}}{\norm{\widecheck{v_n}}_{L^2(G)}} \cdot f, \qquad     
f(x)=\sum_{\chi\in\pi_{\hat G}(\cL_0)}\varphi(\widehat s(\chi))\, \chi(x) \,,
\end{displaymath}
where $f\in C(G)$. Now consider  $f_\varphi\in C(H)$ from Definition~\ref{def:si}, which satisfies
  \begin{align*}
    f_\varphi(y)&
   =
\sum_{\eta\in\pi_{\hat H}(\cL_0)}
      \varphi(\eta)\,\overline{y(\eta)}
 =\sum_{\chi\in\pi_{\hat G}(\cL_0)}
      \varphi(\widehat s(\chi))\, \overline{\widehat s(\chi)(y)} \,.
  \end{align*}
Consider the star map $s:\pi_G(\cL)\to \pi_H(\cL)$, which is well-defined as~$\cL$ projects injectively to~$G$.
Note that~$f$  coincides on $\pi_G(\cL)$ with $f_\varphi\circ s$.
Indeed, for $\lambda\in\pi_G(\cL)$ we have
  \begin{align*}
    f(\lambda)&
    =\sum_{\chi\in\pi_{\hat G}(\cL_0)}
      \varphi(\widehat s(\chi))\, \chi(\lambda)=\sum_{\chi\in\pi_{\hat G}(\cL_0)}
      \varphi(\widehat s(\chi))\, \overline{\widehat s(\chi)(s(\lambda))}= f_\varphi(s(\lambda))
    \,.
  \end{align*}
Here we used that for any $\chi\in\pi_{\hat G}(\cL_0)$
  and $\lambda\in\pi_G(\cL)$,
  we have $\ell_0:=(\chi,\widehat s(\chi))\in\cL_0$
  and $\ell:=(\lambda,s(\lambda))\in\cL$,
  and further
  \begin{equation*}
    1=\ell_0(\ell)=(\chi,\widehat s(\chi))(\lambda,s(\lambda))
      =\chi(\lambda)\widehat s(\chi)(s(\lambda))
    \implies
   \chi(\lambda)=\overline{ \widehat s(\chi)(s(\lambda))}
    \,.
  \end{equation*}
Now we can compute
  \begin{equation}\label{eq2}
    \begin{split}
   \sum_{\lambda\in\Lambda_W}\abs{f_n(\lambda)}^2
      &  = \sum_{\lambda\in\Lambda_W}\abs{f(\lambda)}^2 \frac{\abs{\widecheck{v_n}(\lambda)}^2}{\norm[L^2(G)]{\widecheck{v_n}}^2}
=\frac{1}{\norm[L^2(G)]{\widecheck{v_n}}^2}\sum_{\lambda\in \pi_G(\cL)} \abs{f_\varphi(s(\lambda))}^2 \ifu W(s(\lambda))\abs{\widecheck{v_n}(\lambda)}^2\\
&
=\frac{1}{\norm[L^2(G)]{\widecheck{v_n}}^2}\omega_{\abs{f_\varphi}^2 \ifu W}(\abs{\widecheck{v_n}}^2)  \,.
    \end{split}
  \end{equation}
%
By outer regularity of the Haar measure~$\theta_H$ on~$H$ and by  \cite[Lem.~3.6]{HR15} there are windows $W_\varepsilon \supset W$ such that $\theta_H(\partial W_\varepsilon)=0$ and $\theta_H(W_\varepsilon)-\theta_H(W)\le \varepsilon$ for any $\varepsilon>0$.
We can thus apply the density formula Proposition~\ref{denssm} for unbounded averaging functions to conclude from Eqn.~\eqref{eq2} that
\begin{displaymath}
  \begin{split}
    \limsup_{n\to\infty}  \sum_{\lambda\in\Lambda_W}\abs{f_n(\lambda)}^2
   &=\limsup_{n\to\infty}  \frac{1}{\norm[L^2(G)]{\widecheck{v_n}}^2}\omega_{\abs{f_\varphi}^2 \ifu {W_\varepsilon}}(\abs{\widecheck{v_n}}^2)
    =\dens(\cL)\cdot\int_{W_\varepsilon} \abs{f_\varphi(y)}^2 \dd\theta_H(y)\\
   &\to\dens(\cL)\cdot\int_{W} \abs{f_\varphi(y)}^2 \dd\theta_H(y)
    \qquad (\varepsilon \downarrow 0) \,.
  \end{split}
\end{displaymath}
In order to apply the density formula Proposition~\ref{denssm}, we worked in the bare cut-and-project scheme $(G,H,\cL)$, where we used that~$G$ is $\sigma$-compact and that~$\cL$ projects densely to~$H$.
This combined with Eqn.~\eqref{eq1} shows stable interpolation (i).

%
%
\end{proof}

\begin{proof}[Proof of (ii)]
The strategy of the proof is similar to that of (i). We will associate to any $f\in \CS_K$ suitable $\varphi_n\in \ell^2(_K \Lambda)$ of finite support, compare Fig.~\ref{fig2}.  Our computations will then use diffraction theory of model sets. In particular, we will use commutativity of the modified Wiener diagram, as explained in \cite[Sec.~2, Rem.~5.6]{RS15}.

\smallskip

\begin{figure}
  \begin{equation*}
    \begin{CD}
      \CS_K @> >> L^2(K) \\
      @VVV @VVV \\
      C(H) @< << \ell^2(_K\Lambda)
    \end{CD}
    \qquad\qquad
    \begin{CD}
      f @>>> \widehat f \\
      && @VVV \\
      f_{\varphi_n} @<<< \varphi_n
    \end{CD}
  \end{equation*}
  \caption{Tracing function spaces in proof of the Duality theorem part (ii)}
\label{fig2}
\end{figure}

Consider an arbitrary $f\in \CS_K$. As  $(G,H,\cL)$ is a bare cut-and-project scheme and $\widehat f\in C_c(\widehat G)$, we have that
\begin{displaymath}
_{\abs{\widehat f}^2}\omega= \sum_{(\chi,\eta)\in\cL_0} \abs{\widehat f}^2(\chi) \delta_\eta
\end{displaymath}
is a weighted model comb  in~$\widehat H$. Fix any van Hove sequence $(B_n)_{n\in\mathbb N}$ in~$\widehat H$. Such a sequence indeed exists as~$H$ is assumed to be metrizable. Define $\varphi_n\from\pi_{\widehat H}(\cL_0)\to\mathbb C$ by
\begin{displaymath}
 \varphi_n=\frac{1}{\theta_{\widehat H}(B_n)^{1/2}}
   \ifu{B_n}\cdot(\widehat{f}\circ \widehat t)\,.
\end{displaymath}
Then~$\varphi_n$ is finitely supported as~$\widehat f$ has compact support and~$B_n$ is compact. Here we use that the star map $\hat t\from\pi_{\hat H}(\cL_0)\to\pi_{\hat G}(\cL_0)$ is well-defined since~$\cL_0$ projects injectively to~$\hat H$, which holds by Pontryagin duality as lattice~$\cL$ projects densely to~$G$.
Now the density formula Fact~\ref{thm:df} for~$_{\abs{\widehat f}^2}\omega$ reads
\begin{equation}\label{eq:varphi-density}
 \dens(\cL_0)\cdot\norm[L^2(G)]{f}^2=
 \dens(\cL_0)\cdot\norm[L^2(\widehat G)]{\widehat f}^2 =
\lim_{n\to\infty}   \norm[\ell^2(\pi_{\widehat H}(\cL_0))]{\varphi_n}^2 \,.
\end{equation}
Here we used that~$\widehat H$ is $\sigma$-compact and that~$\cL_0$ projects densely to~$\widehat G$, which follows from injectivity of $\pi_H|_{\cL}$ by Pontryagin duality. Now the stable interpolation lower bound yields an inequality, and we show that the rhs of this inequality is asymptotically bounded from above by the expression needed for stable sampling.

\smallskip

Consider the weighted model comb  $_{\widehat f}\omega= \sum_{(\chi,\eta)\in\cL_0} \widehat f(\chi) \delta_\eta$  in the bare cut-and-project scheme $(\widehat G,\widehat H, \cL_0)$. By standard reasoning, see e.\,g.~\cite[Prop.~5.1]{RS15}, its autocorrelation  is $\gamma=\dens(\cL_0)\cdot _{\widehat f * \widetilde{\widehat f}} \omega$, where we used that~$\cL_0$ projects densely to~$\widehat G$, which follows from injectivity of $\pi_H|_\cL$ by Pontryagin duality.  Moreover by \cite[Thm.~4.10]{RS15}, its  diffraction measure is  $\widehat\gamma=\dens(\cL_0)^2\cdot {_{\abs{f}^2}}\omega$, where $_{\abs{f}^2}\omega= \sum_{(x,y)\in\cL} \abs{f}^2(x) \delta_y$.  Let us consider the normalized finite sample autocorrelations
\begin{displaymath}
  \gamma_n=\frac{1}{\theta_{\widehat H}(B_n)} {_{\widehat f}\omega}|_{B_n} * \widetilde{{_{\reallywidehat f}\omega|_{B_n}}} \,,
\end{displaymath}
%
where $*$ denotes convolution of measures, and where $\widetilde \mu(g)=\mu(\widetilde g)$ denotes measure reflection.
As finite point measures satisfy the convolution theorem \cite[Thm.~3.1]{ARMA1}, we have
\begin{displaymath}
\widehat {\gamma_n}=\frac{1}{\theta_{\widehat H}(B_n)} \cdot \reallywidehat{{_{\widehat f}\omega}|_{B_n}}\,  \cdot \overline{\reallywidehat{{_{\widehat f}\omega}|_{B_n}}}
= \abs{f_{\varphi_n}}^2\theta_H\,,
\end{displaymath}
where multiplication of absolutely continuous measures $f\theta_H$ and $g\theta_H$ is understood as $(f\theta_H)\cdot(g\theta_H)=(fg)\theta_H$. Here  $f_{\varphi_n}: H\to\mathbb C$ is the discrete Fourier transform associated to~$\varphi_n$, i.\,e., the finite trigonometric sum
\begin{displaymath}
 f_{\varphi_n}=\sum_{\eta\in B_n}\varphi_n(\eta) \, \overline{\eta}
=\frac{1}{\theta_{\widehat H}(B_n)^{1/2}}\sum_{(\chi,\eta)\in \cL_0\cap(\widehat G\times B_n)} \widehat{f}(\chi)\, \overline{\eta} \,.
\end{displaymath}
%
%
%
%
%
We have argued that for every $h\in C_c(H)$ we have
\begin{displaymath}
\widehat{\gamma_n}(h)= \int_H \abs{f_{\varphi_n}}^2(y)h(y) \dd\theta_H(y)  \, .
\end{displaymath}
%
%
%
%
%
%
Note that the Fourier transform
is continuous on the space of positive definite complex Radon measures, see  \cite[Lem.~1.26]{MoSt} for the general case and  \cite[Thm.~4.16]{BF} for positive measures.
Since $\lim_{n\to\infty} \gamma_n=\gamma$ vaguely, we thus also have $\widehat \gamma=\lim_{n\to\infty} \widehat{\gamma_n}$ vaguely. Taking $h_\varepsilon\in C_c(H)$ such that $\ifu {\cl W}\le h_\varepsilon$ and $h_\varepsilon\to \ifu {\cl W}$ pointwise as $\varepsilon\downarrow 0$, we can argue
\begin{equation}\label{eq:fvarphin-density}
\begin{split}
\limsup_{n\to\infty} \widehat{\gamma_n}(W) \le \inf_{\varepsilon>0}\limsup_{n\to\infty} \widehat{\gamma_n}(h_\varepsilon)=\inf_{\varepsilon>0} \widehat{\gamma}(h_\varepsilon)=\widehat{\gamma}(\cl W)\,,
\end{split}
\end{equation}
where we used dominated convergence for the last equality.
We thus have shown
\begin{equation}\label{eq:sampint}
\limsup_{n\to\infty} \norm[L^2(H)]{\ifu W f_{\varphi_n}}^2= \limsup_{n\to\infty} \widehat{\gamma_n}(W) \le \widehat{\gamma}(\cl{W})=
\dens(\cL_0)^2\cdot\sum_{\lambda \in \Lambda_{\cl {W}}} \abs{f(\lambda)}^2\ ,
\end{equation}
where in the last equation we again used that $\pi_G|_\cL$ is injective.
We combine this with Eqn.~\eqref{eq:varphi-density} and the definition of stable interpolating:
\begin{align*}
  A_I\cdot\dens(\cL_0)\cdot\norm f_{L^2(G)}^2
  =A_I\cdot\lim_{n\to\infty}\norm{\varphi_n}_{\ell^2(\pi_{\widehat H}(\cL_0))}^2&
  \le\limsup_{n\to\infty}\norm{\ifu Wf_{\varphi_n}}_{L^2(H)}^2\\&
  \le\dens(\cL_0)^2\sum_{\lambda\in\Lambda_{\cl W}}\abs{f(\lambda)}^2\,,
\end{align*}
which finishes the proof that stable sampling is implied by stable interpolation.
\end{proof}

Reinspecting the proof of part (ii) in the duality theorem, the following observations arise.

\begin{remark}\leavevmode
\begin{itemize}
\item[(i)]  The argument in Eqn.~\eqref{eq:fvarphin-density} can analogously be applied for~$W$ an open set. As a result, we have equality in Eqn.~\eqref{eq:sampint} if $\widehat\gamma(\partial W)=0$, a condition which is satisfied if $\partial W\cap \pi_H(\cL)=0$. The latter property holds generically if~$\partial W$ is nowhere dense,  compare Remark \ref{rem:d1} (c).
\item[(ii)] By the Fourier-Bohr coefficient formula \cite[Thm.~5]{L09}, we have for every $y\in H$ that
\begin{displaymath}
\lim_{n\to\infty}  \frac{1}{\theta_{\widehat H}(B_n)} \cdot \abs{f_{\varphi_n}}^2(y)=\widehat\gamma(\{y\})\, ,
\end{displaymath}
which implies $\widehat\gamma(W)=\lim_{n\to\infty}  \frac{1}{\theta_{\widehat H}(B_n)} \cdot \sum_{y\in W\cap\pi_H(\cL)}  \abs{f_{\varphi_n}}^2(y)$. This differs from  $\widehat \gamma(\cl W)\ge \limsup_{n\to\infty} \widehat{\gamma_n}(W)$ in  Eqn.~\eqref{eq:sampint}.
\end{itemize}

\end{remark}

\subsection{Stable sampling and stable interpolation for simple model sets}

Theorem~\ref{thm:main} can now be used in conjunction with Fact~\ref{fact:1d} to give a proof of Theorem~\ref{cor:SIn1}.

\begin{proof}[Proof of Theorem~\ref{cor:SIn1}]
Note that by Eqn.~\eqref{eq:denswms} the weak model set $_K \Lambda \subset \widehat H$ satisfies the density formula
\begin{equation}\label{form:gdf}
\dens(\cL_0) \cdot \theta_{\widehat G}(\inn K)\le \mathcal D_\cA^-(_K\Lambda)\le
 \mathcal D_\cA^+(_K\Lambda) \le \dens(\cL_0) \cdot \theta_{\widehat G}(\cl{K})\, ,
\end{equation}
where~$\cA$ is any van Hove sequence in~$\widehat H$.

\noindent (i)
By assumption we have $\theta_{\widehat G}(\cl K)<\mathcal D(\Lambda_I)=\dens(\cL)\cdot\abs I$, where the latter equality is the density formula for the regular model set~$\Lambda_I$. Choose any Riemann integrable neighborhood~$K'$ of~$\cl K$ such that $\theta_{\widehat G}(K')<\mathcal D(\Lambda_I)$. Such a neighborhood exists by \cite[Lem.~3.6]{HR15}.  Then by the density formula Eqn.~\eqref{form:gdf} we have $\abs I>\mathcal D^+_\cA(_{K'}\Lambda)$.  Choose compact $I'\subset I$ such that $\abs{I'}>\mathcal D^+_\cA(_{K'}\Lambda)$.  Then $_{K'}\Lambda \subset \widehat H \cong\mathbb R$ is stable interpolating for~$\PW_{I'}$ by Fact~\ref{fact:1d} (ii). We can now invoke the Duality Theorem~\ref{thm:main} (ii) to infer that $\Lambda_{\cl{I'}}=\Lambda_{I'}$ is stable sampling for~$\CS_{K'}$. Since~$K'$ is Riemann integrable,~$\Lambda_{I'}$ is stable sampling for~$\PW_{K'}$ by Proposition~\ref{prop:denseinL2}. In particular,~$\Lambda_{I'}$ is stable sampling for~$\PW_{K}$. Since $I'\subset I$, we have $\Lambda_{ I'}\subset \Lambda_{I}$. Hence also~$\Lambda_{I}$ is stable sampling for~$\PW_{K}$. Indeed, in Definition~\ref{def:StSa} the lower inequality continues to hold for any superset of~$\Lambda$, whereas an upper inequality is satisfied for general uniformly discrete sets.

\smallskip

\noindent (ii)
By assumption we have $\theta_{\widehat G}(\inn K)>\mathcal D(\Lambda_I)=\dens(\cL)\cdot\abs I$, where the latter equality is the density formula for the regular model set~$\Lambda_I$.
Now choose an interval $I'\supset I$ such that $\abs{I'}>\abs I$ and such that still $\dens(\cL)\cdot\abs{I'}<\theta_{\widehat G}(\inn K)$. Then by the density formula Eqn.~\eqref{form:gdf} we have $\abs{I'}<\mathcal D^-_\cA(_K\Lambda)$.
Hence $_K\Lambda \subset \widehat H \cong\mathbb R$ is stable sampling for~$\PW_{I'}$ by Fact~\ref{fact:1d} (i).  In particular~$_K\Lambda$ is stable sampling for~$\CS_{I'}$. We can now invoke the Duality Theorem~\ref{thm:main} (i') to infer that~$\Lambda_I$ is stable interpolating for~$\PW_K$.
\end{proof}

\section{Topological properties of the Paley--Wiener space}\label{sec:closed}

\subsection{The test function space~$\CS_K$}

\begin{prop}\label{prop:PWdense}
Let~$G$ be an LCA group and assume that $K\subset \widehat G$ is Riemann measurable.
  Then~$\CS_K$ is dense in $(\PW_K, \norm{\cdot}_{L^2(G)})$.
\end{prop}

\begin{proof}
  We will use the function space $K_2(G)=\lin \{ f*\widetilde{f} : f \in C_c(G) \}$, where $\widetilde f(x)=\overline{f(-x)}$, compare \cite{ARMA1}.
  Take any $f\in \PW_K$.
  For any given $\varepsilon>0$, we define an approximation $f_\varepsilon\in \CS_K$ as follows.
  Take any Dirac net $(v_j)_j$ in~$\widehat G$.
  As $C_c(\widehat G)$ is dense in $L^2(\widehat G)$ and $K_2(\widehat G)$ is dense in
  $C_c(\widehat G)$ and the embedding $C_c(\widehat G)\subset L^2(\widehat G)$ is continuous, also $K_2(\widehat G)$ is dense in $L^2(\widehat G)$.
  Hence we may choose $\varphi\in K_2(\widehat G)$ such that $\norm{\varphi-\widehat f}_2<\varepsilon$.
  Choose any compact unit neighborhood $U\subset \widehat G$ such that $\norm{\widehat f \cdot  \ifu {\partial^UK}}_2<\varepsilon$.
  Such a neighborhood indeed exists:  By H\"older's inequality we have $\norm{\widehat f \cdot  \ifu {\partial^UK}}_2^2\le \norm{\widehat f}_2^2 \cdot \norm{\ifu {\partial^U K}}_2^2=\norm{f}_2^2 \cdot \theta_{\widehat G}(\partial^U K)$. Now the result follows from Riemann measurability of~$K$ and outer regularity of the Haar measure by taking~$U$ sufficiently concentrated around unity, see \cite[Lem.~4.4]{HR15}.
  Choose~$j$ such that $\supp(v_j)\subset U$ and that $\norm{\widehat f-\widehat f *v_j}_2=\norm{f-f \cdot \widehat v_j}_2<\varepsilon$.
  The latter can be achieved as~$\widehat{v_j}$ converges locally to one.
  Consider now
  \begin{displaymath}
    \widehat{f_\varepsilon}:= v_j *(\varphi\cdot \ifu {K\setminus \partial^UK}) \in L^2(\widehat G)\,.
  \end{displaymath}
  Note that in fact $\widehat{f_\varepsilon}\in C_c(\widehat G)$, as it is a convolution of two compactly supported functions in $L^2(\widehat G)$.
  By construction we have $\supp(\widehat{f_\varepsilon})\subset K$. Furthermore, we have $f_\varepsilon:= \widecheck{\widehat{f_\varepsilon}}\in L^1(G)$ since
  \begin{displaymath}
    f_\varepsilon=\widecheck{v_j} \cdot (\widecheck \varphi * \widecheck{\ifu {K\setminus \partial^UK}})
  \end{displaymath}
  is a product of two functions in $L^2(G)$.
  Indeed, the convolution is in $L^2(G)$ by Young's inequality, as $\widecheck\varphi\in L^1(G)$ and $\widecheck{\ifu {K\setminus \partial^UK}}\in L^2(G)$.
  Altogether we  conclude  $f_\varepsilon\in \CS_K$. Now we can estimate
  \begin{displaymath}
  \begin{split}
    \lVert f-{}&f_\varepsilon\rVert_2
    = \norm{\widehat f-\widehat{f_\varepsilon}}_2
    =\norm{\widehat f- v_j *(\varphi\cdot \ifu {K\setminus \partial^UK})}_2\\
   &\le\norm{\widehat f-\widehat f*v_j}_2+\norm{\widehat f*v_j-v_j*(\widehat f\cdot\ifu{K\setminus\partial^UK})}_2+\norm{v_j*(\widehat f\cdot\ifu{K\setminus\partial^UK})-v_j*(\varphi\cdot\ifu{K\setminus\partial^UK})}_2 \\
   &=\norm{\widehat f-\widehat f * v_j}_2 + \norm{v_j *(\widehat f \cdot \ifu {\partial^UK})}_2 + \norm{v_j *((\widehat f-\varphi) \cdot \ifu {K\setminus \partial^UK})}_2\\
   &\le\norm{\widehat f-\widehat f * v_j}_2 + \norm{v_j}_1 \cdot \norm{ \widehat f \cdot  \ifu {\partial^UK} }_2 + \norm{v_j}_1 \cdot\norm{ (\widehat f-\varphi) \cdot \ifu {K\setminus\partial^UK} }_2\\
   &\le\norm{\widehat f-\widehat f*v_j}_2+\norm{ \widehat f \cdot  \ifu {\partial^UK} }_2 + \norm{ \widehat f-\varphi }_2 \le 3\varepsilon\,,
  \end{split}
  \end{displaymath}
  where we used Young's inequality in the fourth line.
  As $\varepsilon>0$ was arbitrary, this proves the claim of the proposition.
\end{proof}

\subsection{Another definition of Paley-Wiener space}

If $K\subset \widehat G$ is closed, then there is another definition of Paley-Wiener space. Recall that 
for a function $f\in L^2(G)$, its support is defined to be
\begin{displaymath}
\supp(f)=\{x\in G: \int_{U} \abs{f(y)}^2 \dd\theta_G(y) >0 \text{ for all open } U \ni x\}\,.
\end{displaymath}
We consider the subspace~$\PWd_K$ of~$\PW_K$ defined by
  \begin{displaymath}
    \PWd_K=\{f\in L^2(G) \mid \spt(\widehat f) \subset K\} \ .
  \end{displaymath}

\begin{prop}\label{prop:closedness}
 Let~$G$ be an LCA group. We have $\PWd_K=\PW_K$ if and only if\/ $\supp(\ifu K)\subset K$.  
  In particular $\PWd_K=\PW_K$ if\/~$K$ is closed in~$\widehat G$.
\end{prop}

A proof will be given at the end of this section after Lemma~\ref{prop:L2closed}.
We first collect some topological properties of the support.

\begin{lemma}\label{lemma:supp}
  Let~$G$ be an LCA group, and let\/ $K\subset G$ be measurable.
  \begin{enumerate}
  \item
    For all $f\in L^2(G)$, $\supp(f)\subset G$ is closed.
  \item\label{item:ae}
    For all $f\in L^2(G)$, we have $f=f\cdot\ifu{\supp(f)}$ almost everywhere.
  \item\label{item:Kclosedsupp1K}
    If~$K$ is closed, then $\supp(\ifu K)\subset K$.
  \item\label{item:KsetminussuppifuK}
    If~$K$ is open or $\theta_G(K)<\infty$ or~$G$ is $\sigma$-finite, then
    $\theta_G(K\setminus\supp(\ifu K))=0$.
  \end{enumerate}
\end{lemma}
\begin{proof}
  \begin{enumerate}[{Ad }(i).,wide]
  \item Fix $f\in L^2(G)$.
    For all $x\in G\setminus\supp(f)$,
    there is an open neighborhood~$U_x$
    with $\norm{f\cdot\ifu U}_2=0$.
    $U_x$ is open, hence $U_x\subset G\setminus\supp(f)$.
    Thus, $G\setminus\supp(f)=\Union_{x\in G\setminus\supp(f)}U_x$ is open.
  \item Let $f\in L^2(G)$.
    The measure $\abs f^2\cdot \theta_G$ is a finite Radon measure.
    Indeed, the map $C_c(G)\to\C$, defined by $g\mapsto\theta_G(\abs f^2\cdot g)$,
    is a continuous linear functional,
    since for all $g\in C_c(G)$ we have
    \begin{equation*}
      \abs*{\theta_G(\abs f^2\cdot g)}
      =\abs*{\int \abs f^2 g\dd\theta_G}
      \le C\cdot \norm g_\infty\,,
    \end{equation*}
    where $C=\norm f_2^2<\infty$ is independent of~$g$.
    By the Riesz representation theorem,
    $B\mapsto\theta_G(\abs f^2\cdot \ifu B)
     =\norm{f\cdot \ifu B}_2^2$
    is a Radon measure.
    \par
    We now turn to the support of~$f$.
    Each $x\in G\setminus\supp(f)$ has an open neighborhood~$U_x$
    such that $\norm{f\cdot \ifu{U_x}}_2=0$.
    Each compact $C\subset G\setminus\supp(f)$
    admits a finite subset $I_C\subset C$
    indexing the finite subcover $\Union_{x\in I_C}U_x\supset C$.
    Thus, $\norm{f\cdot \ifu C}_2^2\le\sum_{x\in I_C}\norm{f\cdot \ifu{U_x}}_2^2=0$.
    We now use regularity to conclude
    \begin{align*}
      \norm{f-f\cdot\ifu{\supp(f)}}_2^2&
      =\norm{f\cdot\ifu{G\setminus\supp(f)}}_2^2
      =\sup_{\substack{C\subset G\setminus\supp(f),\\
        \text{$C$ compact}}}\norm{f\cdot \ifu C}_2^2
      =0\,\text.
    \end{align*}
  \item
    For closed~$K\subset G$, the set $U=G\setminus K$ is open,
    and we have $\norm{\ifu K\cdot\ifu U}_2=0$.
    By definition of~$\supp$, we conclude $U\subset G\setminus\supp(\ifu K)$.
    This is equivalent to $\supp(\ifu K)\subset K$.
  \item Note that each
    $x\in K\setminus\supp(\ifu K)$ has an open neighborhood~$U_x$
    with $\norm{\ifu{U_x}\cdot\ifu K}_2=0$.
    For compact $C\subset K\setminus\supp(\ifu K)$,
    there is a finite set~$I_C\subset C$
    indexing a covering $\Union_{x\in I_C}U_x\supset C$.
    We conclude
    \begin{equation*}
      \theta_G(C)
      =\theta_G(C\isect K)
      \le\sum_{x\in I_C}\theta_G(U_x\isect K)
      =\sum_{x\in I_C}\norm{\ifu{U_x}\cdot\ifu K}_2^2
      =0\,\text.
    \end{equation*}
    The Haar measure~$\theta_G$ is inner regular,
    so we have for all open~$K$ and all~$K$ with $\theta_G(K)<\infty$
    \begin{equation*}
      \theta_G(K\setminus\supp(\ifu K))
      =\sup_{\substack{C\subset K\setminus\supp(\ifu K)\\
        \text{$C$ compact}}}\theta_G(C)
      =0\,\text.
    \end{equation*}
    If~$G$ is $\sigma$-finite, we exhaust~$K\setminus\supp(\ifu K)$
    with countably many sets of finite measure.
    \qedhere
  \end{enumerate}
\end{proof}


Let $K\subset G$ be measurable. We denote by $L^2(G:K)$
the subspace of functions having support within~$K$. In slight abuse of notation, we also write
 \begin{displaymath}
    L^2(K)=\{f\in L^2(G)\mid\norm{f\cdot\ifu{G\setminus K}}_2=0\}=\{f\in L^2(G) \mid  f |_{K^c}=0 \text{ almost everywhere} \}
  \end{displaymath}
The relation between these spaces is described by the following result.

\begin{lemma}\label{prop:L2closed}
  Let~$G$ be an LCA group and let\/ $K\subset G$ be measurable. Then the following hold.
  \begin{enumerate}[(a)]
  \item\label{item:a} We have
    \begin{equation*}
      \cl{L^2(G:K)} =L^2(K) \subset L^2(G:\supp(\ifu K))
    \end{equation*}
    with equality if and only if $\theta_G(\supp(\ifu K)\setminus K)=0$.
  \item\label{item:b} The following are equivalent:
    \begin{enumerate}[label=\textit{(\roman*)}]
    \item\label{item:suppKsubset} $\supp(\ifu K)\subset K$.
    \item\label{item:L2GKclosed} $L^2(G:K)$ is a closed subspace of\/ $L^2(G)$.
    \end{enumerate}
  \end{enumerate}
   In particular, if\/~$K$ is closed in~$G$,
    then $L^2(G:K)$ is closed in~$L^2(G)$, and the spaces in (a) coincide.
\end{lemma}

\begin{proof}
  We begin with~\ref{item:a}.
  \begin{itemize}[wide]
  \item
    \emph{$\cl{L^2(G:K)}\subset L^2(K)$:}
    Since $L^2(G)\ni f\mapsto\norm{f\cdot\ifu{G\setminus K}}_2\in\R$
    is continuous, the right hand side is closed.
    Fix $f\in L^2(G:K)$.
    From Lemma~\ref{lemma:supp}~\ref{item:ae}\ and $\supp(f)\subset K$, we get
    $\norm{f\cdot\ifu{G\setminus K}}_2
      =\norm{f\cdot\ifu{\supp(f)}\cdot\ifu{G\setminus K}}_2
      =0$, which establishes the claim.
  \item
    \emph{$\cl{L^2(G:K)}\supset L^2(K)$:}
    Fix $f\in L^2(K)$.
    Using the inner regularity of the Radon measure $\abs f^2\cdot\theta_G$,
    compare the proof of Lemma~\ref{lemma:supp}~\ref{item:ae},
    we see that
    \begin{equation*}
      \norm f_2^2
      =\norm{f\cdot\ifu K}_2^2+\norm{f\cdot\ifu{G\setminus K}}_2^2
      =\sup_{\substack{C\subset K\\\text{$C$ compact}}}\norm{f\cdot\ifu C}_2^2
      =\lim_{n\to\infty}\norm{f\cdot\ifu{C_n}}_2^2
    \end{equation*}
    for a suitable sequence of compact~$C_n\subset K$, $n\in\N$.
    For all $n\in\N$, we know by Lemma~\ref{lemma:supp}~\ref{item:Kclosedsupp1K}
    that $\supp(\ifu{C_n})\subset C_n\subset K$ since~$C_n$ is closed,
    so $f\cdot\ifu{C_n}\in L^2(G:K)$.
    Now
    \begin{equation*}
      \norm{f-f\cdot\ifu{C_n}}_2^2
      =\int\abs f^2\cdot(1-\ifu{C_n})\dd\theta_G
      =\norm f_2^2-\norm{f\cdot\ifu{C_n}}_2^2
      \xto{n\to\infty}0\,\text.
    \end{equation*}
  \item \emph{$L^2(K)\subset L^2(G:\supp(\ifu K))$:}
    Fix $f\in L^2(G)$ with $\norm{f\cdot\ifu{G\setminus K}}_2=0$
    and define $U=G\setminus\supp(\ifu K)$.
    Note that~$U$ is open and satisfies $\norm{\ifu U\cdot\ifu K}_2=0$.
    Then
    \begin{equation*}
      \norm{f\cdot\ifu U}_2 = \norm{f\cdot\ifu K\cdot\ifu U}_2+\norm{f\cdot\ifu{G\setminus K}\cdot\ifu U}_2
      =0\,\text.
    \end{equation*}
    Therefore, $U\subset G\setminus\supp(f)$, which is equivalent to $\supp(f)\subset\supp(\ifu K)$.
  \item \emph{If $\theta_G(\supp(\ifu K)\setminus K)=0$, then
    $L^2(K)\supset L^2(G:\supp(\ifu K))$:}
    Fix $f\in L^2(G:\supp(\ifu K))$ and note that
    $
      \supp(f)\isect(G\setminus K)
      \subset\supp(\ifu K)\setminus K
    $
    is by assumption a set of Haar measure zero.
    Therefore, by Lemma~\ref{lemma:supp}~\ref{item:ae},
    $\norm{f\cdot\ifu{G\setminus K}}_2
      =\norm{f\cdot\ifu{\supp(f)}\cdot\ifu{G\setminus K}}_2
      =0$.
  \item \emph{If\/ $\theta_G(\supp(\ifu K)\setminus K)>0$, then
    $L^2(K)\ne L^2(G:\supp(\ifu K))$:}
    By inner regularity of the Haar measure~$\theta_G$,
    we find a compact $C\subset\supp(\ifu K)\setminus K$
    with $\theta_G(C)>0$.
    Define $f=\ifu C \in L^2(G)$ and note
    $\norm{f\cdot\ifu{G\setminus K}}_2
      =\norm{f}_2
      =\theta_G(C)>0$,
    but by Lemma~\ref{lemma:supp}~\ref{item:Kclosedsupp1K}
    $\supp f
      \subset C
      \subset\supp(\ifu K)\setminus K
      \subset\supp(\ifu K)$.
  \end{itemize}

\smallskip 

  For the proof of~\ref{item:b}, we need some more prerequisites. Note first that we may assume $\theta_G(K)>0$
without loss of generality. Indeed, otherwise $\supp \ifu K=\varnothing$. This implies $L^2(G:K)=\{f\in L^2(G): f=0 \text{ almost everywhere}\}$, and the claim is trivially satisfied.
  \begin{enumerate}[1.,wide,resume]
  \item\label{item:K'suppifuK}
    \emph{Let $K':=\Union\limits_{f\in L^2(G:K)}\supp(f)$.
    Then $\cl{K'}=\supp(\ifu K)$:}
    From~\ref{item:a} we know that $K'\subset\supp(\ifu K)$,
    which implies $\cl{K'}\subset\supp(\ifu K)$ since $\supp(\ifu K)$ is closed.
    For the reverse inclusion, fix $x\in\supp(\ifu K)$
    and any open relatively compact neighborhood~$U$ of~$x$.
    As $K\isect U$ has finite measure, we can use inner regularity of the Haar measure to write
    \begin{equation*}
      0<\norm{\ifu K\cdot\ifu U}_2
      =\sup_{\substack{C\subset K\isect U\\\text{$C$ compact}}}\norm{\ifu C}_2\,.
    \end{equation*}
    Therefore, there exists a compact set $C\subset K\isect U$
    such that $\norm{\ifu C}_2>0$.
    We conclude from Lemma~\ref{lemma:supp}~\ref{item:KsetminussuppifuK}
    that $\theta_G(\supp(\ifu C))\ge\theta_G(C)>0$, so that $\supp(\ifu C)\ne\varnothing$.
    Because~$C$ is closed, we have $\supp(\ifu C)\subset C\subset K\isect U\subset K$,
    so $\ifu C\in L^2(G:K)$.
    Consequently $\varnothing\neq\supp(\ifu C)\subset K'\isect U$. Hence
    any open relatively compact neighborhood of~$x$ has nontrivial intersection with~$K'$.
    This implies $x\in\cl{K'}$, and since $x\in\supp(\ifu K)$
    was arbitrary, we have $\supp(\ifu K)\subset\cl{K'}$.
  \item\label{item:iimpliesK'closed}
    \emph{\ref{item:L2GKclosed} implies that $K'$ 
    is closed:}
    Assume $x_n\in K'$, $n\in\N$, and $x\in G$ such that $x_n\to x$.
    We show $x\in K'$.
    By definition of~$K'$, for each~$x_n\in K'$
    there exists $f_n\in L^2(G:K)$ such that $x_n\in\supp( f_n)$.
    In particular we have $\norm{f_n}_2>0$ for all~$n$.
    Hence we can define $f\in L^2(G)$ by
    \begin{equation*}
      f:=\sum_{n\in\N}\frac{\abs{f_n}}{2^n\norm{f_n}_2}
      \qtextq{using}
      \norm{f}_2
      \le\sum_{n\in\N}\frac{\norm{f_n}_2}{2^n\norm{f_n}_2}
      <\infty\,.
    \end{equation*}
    By assumption~\ref{item:L2GKclosed} that~$L^2(G:K)$ is closed,
    we have $f\in L^2(G:K)$.
    Now $x_n\in\supp(f)$ for every~$n$,
    because for all~$n$, we have $f\ge\frac{\abs{f_n}}{2^n\norm{f_n}_2}$,
    and hence for any open neighborhood~$U$ of~$x_n$
    \begin{equation*}
      \norm{f\cdot\ifu U}_2
      \ge\frac{\norm{f_n\cdot \ifu U}_2}{2^n\norm{f_n}_2}
      >0
      \,\text.
    \end{equation*}
    The support of~$f$ is closed, so this implies
    $x\in\supp(f)$. As $\supp(f)\subset K'$, this shows that~$K'$
    contains~$x$ and is therefore closed.
  \end{enumerate}
  With this preparation, the proof of~\ref{item:b} reduces to two lines.
  \begin{itemize}[wide]
  \item \emph{\ref{item:suppKsubset}$\implies$\ref{item:L2GKclosed}:}
    By~\ref{item:a} and~\ref{item:suppKsubset}, we have
    $\cl{L^2(G:K)}
      \subset L^2(G:\supp(\ifu K))
      \subset L^2(G:K)
    $.
  \item \emph{\ref{item:L2GKclosed}$\implies$
    \ref{item:suppKsubset}:}
    Use \ref{item:K'suppifuK}, \ref{item:iimpliesK'closed}, and
    the definition of~$K'$ for $\supp(\ifu K)=\cl{K'}=K'\subset K$.
    \qedhere
  \end{itemize}
\end{proof}

\begin{proof}[Proof of Proposition~\ref{prop:closedness}]
  As the Fourier transform is a norm isomorphism between $L^2(G)$ and $L^2(\widehat G)$,
  see e.\,g.~\cite[Thm.~4.4.13]{Rei2}, the space $\PWd_K$ is closed if and only if
  $L^2(\widehat G:K)$ is closed.
  Now the claim is obvious from Lemma~\ref{prop:L2closed}.
\end{proof}

\section{Van Hove Sequences and Banach densities}\label{sec:vHBd}

We collect some properties of van Hove boundaries $\partial^KA=((K+\cl{A})\cap \cl{A^c}) \cup ((K+\cl{A^c})\cap \cl{A})$  that will be used in the sequel. As the following proof shows, Lemma~\ref{lem:vHb} continues to hold in the non-abelian setting.

\begin{lemma}\label{lem:vHb}
Let~$G$ be an LCA group, and consider arbitrary sets $A,K,L\subset G$.  We then have
\begin{itemize}
\item[(i)] $K+\cl{A}\subset \inn{A} \cup \partial^K A \subset (K\cup \{0\})+\cl{A}$
\item[(ii)] $\inn{A}\setminus \partial^K A \subset \{x\in  \inn{A}: -K+x \subset \inn{A}\}$
\item[(iii)] $\partial^K(A+L)\subset (\partial^K A) + \cl{L}$
\item[(iv)] $(L+(\partial^K A))\subset (\partial^{L+K}A) \cup (\partial^L A)$
\end{itemize}
\end{lemma}

\begin{proof}
We will use multiplicative notation for the group operation, which is more compact than additive notation.

\noindent ``(i)'' Assume that $x\in K\cl{A}$ but $x\notin \inn{A}$. Then 
$x\in K\cl{A} \cap \inn{A}^c=K\cl{A}\cap \cl{A^c}\subset \partial^KA$. The second statement is obvious.

\noindent ``(ii)'' We compute
\begin{displaymath}
\begin{split}
\inn{A} \setminus \partial^K A &= \inn{A}\setminus   \left((K\cl{A}\cap \cl{A^c}) \cup (K\cl{A^c}\cap\cl{A}) \right) \\
 &= \inn{A}\cap   ((K\cl{A})^c\cup \inn{A}) \cap ((K\cl{A^c})^c\cup\inn{A^c})\\
&= (\inn{A}\cap (K\cl{A^c})^c) \cup (\inn{A} \cap (K\cl{A})^c \cap (K\cl{A^c})^c)= \inn{A}\cap (K\cl{A^c})^c 
\end{split}
\end{displaymath}
Now the claim follows: Assume $k^{-1}x\notin \inn{A}$ for some $k\in K$ and some $x\in \inn{A} \setminus \partial^K A$. Then $x\in k\inn{A}^c\subset K\cl{A^c}$, which is a contradiction.

\noindent ``(iii)'' Let us first consider $x\in K\cl{AL}\cap (\inn{AL})^c$. Then $x=ka\ell$ for some $k\in K$, $a\in \cl{A}$ and $\ell\in \cl{L}$. We show $ka \in \inn{A}^c$, which results in $x\in (K\cl{A} \cap \inn{A}^c) \cl{L}\subset (\partial^KA)\cl{L}$. Assume $ka=x\ell^{-1} \in \inn{A}$. Then $x\in \inn{A}\ell\subset \inn{A}\cl{L} \subset \inn{AL}$, which is a contradiction. It remains to consider the case $x\in K (\inn{AL})^c \cap \cl{AL}$. Then $x=a\ell$ for some $a\in \cl{A}$ and some $\ell\in \cl{L}$. We show $a=x\ell^{-1}\in K\inn{A}^c$, which proves the claim. Indeed, assume $x\ell^{-1} \notin K \inn{A}^c$. Then $x\notin K(\inn{A}\ell)^c \supset K (\inn{A}\cl{L})^c \supset K (\inn{AL})^c$, which is a contradiction.
 
\noindent ``(iv)'' This is proved in \cite[Lem.~2.13]{mr13}.
\end{proof}

The following lemma shows that van Hove sequences may be nested and may have simple topological properties. In fact they may be supersets of any relatively compact seed.

\begin{lemma}\label{lem:vHs}
Let~$G$ be a $\sigma$-compact LCA group with Haar measure~$\theta_G$. Consider any relatively compact set $A\subset G$.
Then there exists a van Hove sequence $(A_n)_{n}$ such that $A_n=\cl{\inn{A_n}}$,  $\theta_G(\partial A_n)=0$, $A_n\subset \inn{A_{n+1}}$ and $A\subset A_n$ for all~$n$.
\end{lemma}

\begin{proof}
(i) We first prove the assertion of the lemma for $A=\{e\}$. Take any van Hove sequence $(C_n)_n$. We  will first use $(C_n)_n$ to construct a van Hove sequence $(B_n)_n$ satisfying $B_n=\cl{\inn{B_n}}$ and $\theta_G(\partial B_n)=0$. Passing to a suitable shifted subsequence will then yield $(A_n)_n$ with the claimed properties. Take any open relatively compact zero neighborhood $U\subset G$. For every~$n$, take a compact set~$B_n$ such that $C_n\subset B_n\subset UC_n$, $\theta_G(\partial B_n)=0$ and $\cl{\inn{B_n}}=B_n$. This is possible by the construction given in the proof of \cite[Lem.~3.6]{HR15}. Consider any compact $K\subset G$. We can now estimate
\begin{displaymath}
\begin{split}
\partial^K B_n &= (K\cl{B_n} \cap \cl{B_n^c}) \cup  (K\cl{B_n^c} \cap \cl{B_n}) \\
&\subset 
 (K\cl{U}\cl{C_n} \cap \cl{C_n^c}) \cup  (K\cl{C_n^c} \cap \cl{U}\cl{C_n}) \\
&\subset \partial^{K\cl{U}}C_n \cup  ( K\cl{C_n^c} \cap ((\cl{C_n} \cap \cl{U}\cl{C_n})\cup( \cl{C_n^c} \cap \cl{U}\cl{C_n}))\\
&  \subset \partial^{K\cl{U}}C_n \cup  \partial^K C_n \cup \partial^{\cl{U}}C_n \ .
\end{split}
\end{displaymath}
Since $\theta_G(B_n)\ge \theta_G(C_n)$, the van Hove property $\theta_G(\partial^K B_n)/\theta_G(B_n)\to 0$ follows from the van Hove property of $(C_n)_n$. Next, we construct $(A_n)_n$ as follows. Take $b\in B_1$ and define $A_1=B_1b^{-1}$. Now assume that $A_1, \ldots, A_k$ satisfying the properties claimed in the statement of the lemma have been chosen. We construct $A_{k+1}$ following the proof of \cite[Lem.~2.6 (e)]{PS16}. Take~$n$ such that $\theta_G(\partial^{A_k^{-1}}B_n)< \theta_G(\inn{B_n})=\theta_G(B_n)$, which is possible due to the van Hove property of $(B_n)_n$. Then $\inn{B_n}\setminus \partial^{A_k^{-1}}B_n \subset \{x\in \inn{B_n}: A_kx\subset \inn{B_n}\}=:S_n$ by Lemma~\ref{lem:vHb} (ii) and
\begin{displaymath}
\theta_G(\inn{B_n}\setminus \partial^{A_k^{-1}}B_n) 
\ge \theta_G(\inn{B_n})-\theta_G(\partial^{A_k^{-1}}B_n) >0 \ .
\end{displaymath}
In particular, we may choose $s\in S_n$ and define $A_{k+1}=B_ns^{-1}$. Then $A_ks\subset \inn{B_n}$ implies $A_k\subset \inn{B_ns^{-1}}=\inn{A_{k+1}}$. It is obvious that $(A_n)_n$ is van Hove, since it is a shifted subsequence of the van Hove sequence $(B_n)_n$.

\noindent (ii) Now choose any relatively compact set $A\subset G$, which we assume to be nonempty without loss of generality.  Take any van Hove sequence $(B_n)_n$ having the properties in part (i) of the proof. Then the sequence $(A_n)_n$ of compact sets defined by $A_n=B_n\cl{A}$ satisfies $A\subset A_n$ and $A_n\subset \inn{A_{n+1}}$ for all~$n$. Take any compact $K\subset G$. By Lemma~\ref{lem:vHb}  and commutativity of~$G$, we can estimate
\begin{displaymath}
\theta_G(\partial^K A_n) \le \theta_G((\partial^KB_n) \cl{A}) = \theta_G(\cl{A}(\partial^KB_n) ) \le 
\theta_G((\partial^{\cl{A} K}B_n) ) + \theta_G((\partial^{\cl{A}} B_n) ) \ .
\end{displaymath}
This can be used to conclude that $(A_n)_n$ is van Hove. Indeed, we have
\begin{displaymath}
\frac{\theta_G(\partial^K A_n)}{\theta_G(A_n)} \le \frac{\theta_G(B_n)}{\theta_G(B_n\cl{A})} \cdot 
\left( \frac{\theta_G(\partial^{\cl{A} K}B_n)  }{\theta_G(B_n)}  + 
\frac{\theta_G(\partial^{\cl{A}}B_n)  }{\theta_G(B_n)}  \right) \to 0 \qquad (n\to \infty) \ .
\end{displaymath}
Here we used that the first factor on the rhs is bounded by~$1$ uniformly in~$n$. Indeed, by translation invariance of the Haar measure we may assume $e\in A$ without loss of generality. But then $B_n\subset B_n\cl{A}$, implying $\theta_G(B_n)/\theta_G(B_n\cl{A})\le1$.
\end{proof}

Following \cite{GKS08}, we consider a transitive relation $\preceq$ between non-negative translation bounded measures. Let $\mu, \nu\in \mathcal{M}^\infty(G)$ such that $\mu,\nu\ge0$. Then $\mu\preceq\nu$ iff for every $\varepsilon>0$ there exists compact $K\subset G$ such that for all compact $A\subset G$ we have $(1-\varepsilon) \mu(A)\le \nu(K+A)$.  Let us write $\nu\succeq\mu$ iff $\mu\preceq\nu$. We have the following result, which allows to  rephrase Theorem 1'' from \cite{GKS08} using upper and lower Banach densities.

\begin{lemma}\label{lem:BanvH}
Let~$G$ be a $\sigma$-compact LCA group with Haar measure~$\theta_G$. Let $\mathcal A=(A_n)_{n\in\mathbb N}$ be any van Hove sequence in~$G$. Then for any uniformly discrete set $\Lambda\subset G$ with associated counting measure $\delta_\Lambda$ and any $\alpha\ge 0$ the following implications hold.
\begin{itemize}
\item[(i)] $\delta_\Lambda \succeq \alpha\theta_G \Longrightarrow  \cd^-_\cA(\Lambda)\ge \alpha$
\item[(ii)] $\delta_\Lambda\preceq \alpha \theta_G \Longrightarrow  \cd^+_\cA(\Lambda) \le \alpha$
\end{itemize}
\end{lemma}

\begin{proof}
``(i)'': Consider arbitrary $\varepsilon>0$. By assumption there exists compact $K\subset G$ such that for all compact $A\subset G$ we have
$(1-\varepsilon)\alpha \theta_G(A)\le \delta_\Lambda(K+A)$. Letting $A=A_n+s$, we have
\begin{displaymath}
(1-\varepsilon)\alpha \le \frac{\delta_\Lambda(K+A_n+s)}{\theta_G(A_n+s)} \le \frac{\delta_\Lambda(A_n+s)}{\theta_G(A_n)} + \frac{\delta_\Lambda(\partial^K(A_n+s))}{\theta_G(A_n)} \ ,
\end{displaymath}
where we used Lemma~\ref{lem:vHb} (i) and translation invariance of the Haar measure.
Now the second term on the rhs vanishes  as $n\to\infty$ uniformly in~$s$ by \cite[Lem.~9.2.(b)]{LR}. This implies $(1-\varepsilon)\alpha \le \cd^-_\cA(\Lambda)$. As $\varepsilon>0$ was arbitrary, the claim follows.

\smallskip

\noindent ``(ii)'': Consider arbitrary $\varepsilon>0$. By assumption there exists compact $K\subset G$ such that for all compact $A\subset G$ we have
$(1-\varepsilon)\delta_\Lambda(A)\le \alpha \theta_G(K+A)$. Letting $A=A_n+s$, we have
\begin{displaymath}
\frac{\alpha}{1-\varepsilon} \ge \frac{\delta_\Lambda(A_n+s)}{\theta_G(K+A_n+s)} = 
\frac{\delta_\Lambda(A_n+s)}{\theta_G(A_n)} \cdot \frac{\theta_G(A_n)}{\theta_G(K+A_n)} \ ,
\end{displaymath}
where we used translation invariance of the Haar measure.
Now the second factor on the rhs tends to unity as $n\to\infty$ by the van Hove property. Indeed, we may assume $0\in K$ and can then use Lemma~\ref{lem:vHb} (i). This implies $\alpha \ge (1-\varepsilon) \cd^+_\cA(\Lambda)$. As $\varepsilon>0$ was arbitrary, the claim follows.
\end{proof}


\begin{thebibliography}{99}

\bibitem{AAC15}
\newblock E.~Agora, J.~Antezana and C.~Cabrelli,
\newblock \textit{Multi-tiling sets, Riesz bases, and sampling near the critical density in LCA groups},
\newblock Adv.~Math.~\textbf{285} (2015), 454--477. 

\bibitem{AACB16}
\newblock E.~Agora, J.~Antezana, C.~Cabrelli and B.~Matei,
\newblock  \textit{Existence of quasicrystals and universal stable
sampling and interpolation in LCA groups}, 
\newblock Trans. Amer. Math. Soc. \textbf{372} (2019), 4647--4674.

\bibitem{ARMA1} L.N.~Argabright and J.~Gil de Lamadrid, \textit{Fourier Analysis of Unbounded Measures on Locally Compact Abelian Groups}, Memoirs of
the Amer. Math. Soc. \textbf{145} (1974).


\bibitem{BG2}
\newblock M.~Baake and U.~Grimm,
\newblock \textit{Aperiodic Order. Vol. 1. A Mathematical Invitation},
\newblock Encyclopedia of Mathematics and its Applications \textbf{149},
\newblock Cambridge University Press, Cambridge (2013).

\bibitem{BF}
C.~Berg and G.~Forst, \textit{Potential Theory on Locally Compact
Abelian Groups}, Springer, Berlin (1975).

\bibitem{Bou}
\newblock N.~Bourbaki, 
\newblock \textit{General Topology, Chapters 5--10},
\newblock Elements of Mathematics (Berlin), Springer, Berlin (1998).

\bibitem{DE} A.~Deitmar and S.~Echterhoff, \textit{Principles of Harmonic Analysis}, Springer, New York (2009).

\bibitem{DHZ19} 
\newblock T.~Downarowicz, D.~Huczek and G.~Zhang, 
\newblock \textit{Tilings of amenable groups},  
\newblock J. Reine Angew. Math. \textbf{747} (2019), 277--298.


\bibitem{GL16}
\newblock S.~Grepstad and N.~Lev,
\newblock \textit{Riesz bases, Meyer's quasicrystals, and bounded remainder sets},
\newblock  Trans. Amer. Math. Soc. \textbf{370} (2018), 4273--4298.

\bibitem{GKS08}
K.~Gr\"ochenig, G.~Kutyniok and K.~Seip,
\textit{Landau's necessary density conditions for LCA groups},
J. Funct. Anal. \textbf{255} (2008), 1831--1850.

\bibitem{HeRo}
E.~Hewitt  and K.A.~Ross,
\textit{Abstract Harmonic Analysis. Vol. I.}, Springer, Berlin (1979).

\bibitem{HR15}
\newblock C.~Huck and C.~Richard,
\newblock \textit{On pattern entropy of weak model sets},
\newblock Discrete~Comput.~Geom.~\textbf{54} (2015), 741--757.

\bibitem{KR18}
\newblock G.~Keller and C.~Richard,
\newblock \textit{Periods and factors of weak model sets}, Israel~J.~Math.~\textbf{229} (2019), 85--132.

\bibitem{KL11}
\newblock G.~Kozma and N.~Lev,
\newblock \textit{Exponential Riesz bases, discrepancy of irrational rotations and BMO},
\newblock J. Fourier Anal. Appl. \textbf{17} (2011), 879--898.

\bibitem{L67}
\newblock H.J.~Landau,
\newblock \textit{Necessary density conditions for sampling and interpolation of certain entire functions},
\newblock  Acta Math.\textbf{117} (1967), 37--52.

\bibitem{L09}
D.~Lenz, \textit{Continuity of eigenfunctions of uniquely ergodic dynamical systems and intensity of Bragg peaks},
Comm. Math. Phys. \textbf{287} (2009), 225--258.

\bibitem{LR}
D.\ Lenz and C.\ Richard, \textit{Pure point diffraction and cut and project schemes for measures: the smooth case},
Math. Z. \textbf{256} (2007), 347--378.

\bibitem{mm10}
\newblock B.~Matei and Y.~Meyer,
\newblock \textit{Simple quasicrystals are sets of stable sampling},
\newblock Complex Var. Elliptic Equ. \textbf{55} (2010), 947--964.

\bibitem{me70}
\newblock Y.~Meyer,
\newblock \textit{Nombres de Pisot, Nombres de Salem et Analyse Harmonique},
\newblock Lecture Notes in Mathematics \textbf{117},
\newblock Springer, Berlin-New York (1970).

\bibitem{me72}
\newblock Y.~Meyer,
\newblock \textit{Algebraic Numbers and Harmonic Analysis},
\newblock North-Holland, Amsterdam (1972).

\bibitem{me73}
\newblock Y.~Meyer,
\newblock \textit{Adeles et series trigonometriques speciales},
\newblock Annals of Mathematics \textbf{97} (1973), 171--186.

\bibitem{me18}
\newblock Y.~Meyer,
\newblock \textit{Mean-periodic functions and irregular sampling},
\newblock Trans.~R.~Norw.~Soc.~Sci.~Lett.  (2018), 5--23.

\bibitem{RVM3}
R.V.~Moody, \textit{Meyer sets and their duals, in: The Mathematics of Long-Range
Aperiodic Order}, (R. V. Moody, ed.), NATO ASI Series \textbf{C 489},
Kluwer, Dordrecht (1997), 403-441.

\bibitem{MoSt}
\newblock R.V.~Moody and N.~Strungaru, 
\newblock \textit{Almost periodic measures and their Fourier transform},
\newblock  Aperiodic order, Vol. 2, 173--270, Encyclopedia Math. Appl. \textbf{166}, Cambridge Univ. Press, Cambridge, 2017.

\bibitem{mr13} 
\newblock P.~M\"uller and C.~Richard,
\newblock \textit{Ergodic properties of randomly coloured point sets},
\newblock Canad.~J.~Math.~\textbf{65} (2013), 349--402.


\bibitem{PS16}
\newblock  F.~Pogorzelski and F.~Schwarzenberger,
\newblock \textit{A Banach space-valued ergodic theorem for amenable groups and applications},
\newblock J. Anal. Math. \textbf{130} (2016), 19--69.

\bibitem{Rei2}
H.~Reiter and  J.D.~Stegeman, {\em Classical Harmonic Analysis and Locally Compact Groups}, Clarendon Press, Oxford (2000).

\bibitem{RS15}
C.~Richard und N.~Strungaru,
\newblock \textit{Pure point diffraction and Poisson summation},
\newblock Ann.~Henri Poincar\'e \textbf{18} (2017), 3903--3931.

\bibitem{RS17}
\newblock C.~Richard and N.~Strungaru,
\newblock \textit{A short guide to pure point diffraction in cut-and-project sets},
\newblock J.~Phys.~A: Math.~Theor. \textbf{50} (2017), 154003.

\bibitem{Martin2}
M.~Schlottmann, \textit{Generalized model sets and dynamical
systems}, in: {\em Directions in Mathematical Quasicrystals}, CRM Monogr. Ser., Amer. Math. Soc., Providence, RI (2000), 143--159.

\bibitem{Y01}
\newblock R.M.~Young,
\newblock \textit{An introduction to nonharmonic Fourier series},
\newblock  revised first edition, Academic Press, New York-London (2001).


\end{thebibliography}
\end{document}